\documentclass{amsart}

\usepackage[all]{xy}
\usepackage{amsmath}
\usepackage{hyperref}
\usepackage{amsfonts,graphics,amsthm,amsfonts,amscd,latexsym}
\usepackage{epsfig}
\usepackage{flafter}
\usepackage{mathtools}
\usepackage{soul}

\hypersetup{
    colorlinks=true,    
    linkcolor=blue,          
    citecolor=blue,      
    filecolor=blue,      
    urlcolor=blue           
}
\usepackage{tikz}
\usetikzlibrary{graphs,positioning,arrows,shapes.misc,decorations.pathmorphing}

\tikzset{
    >=stealth,
    every picture/.style={thick},
    graphs/every graph/.style={empty nodes},
}

\tikzstyle{vertex}=[
    draw,
    circle,
    fill=black,
    inner sep=1pt,
    minimum width=5pt,
]
\usepackage[position=top]{subfig}
\usepackage[backrefs]{amsrefs}
\usepackage{amssymb}
\usepackage{color}

\calclayout

\usetikzlibrary{decorations.pathmorphing}
\tikzstyle{printersafe}=[decoration={snake,amplitude=0pt}]

\newcommand{\codim}{\operatorname{codim}}

\newcommand{\Spec}{\operatorname{Spec}}

 \newcommand{\rddown}[1]{\left\lfloor{#1}\right\rfloor}

\renewcommand{\qq}{\mathbb{Q}}
\newcommand{\zz}{\mathbb{Z}}
\newcommand{\nn}{\mathbb{N}}
\newcommand{\rr}{\mathbb{R}}

\newcommand{\bQ}{\mathbb{Q}}

\def\O#1.{\mathcal {O}_{#1}}			
\def\pr #1.{\mathbb P^{#1}}				
\def\af #1.{\mathbb A^{#1}}				
\def\ses#1.#2.#3.{0\to #1\to #2\to #3 \to 0}		
\def\xrar#1.{\xrightarrow{#1}}			
\def\K#1.{K_{#1}}						
\def\bA#1.{\mathbf{A}_{#1}}				
\def\bM#1.{\mathbf{M}_{#1}}				
\def\bL#1.{\mathbf{L}_{#1}}				
\def\bB#1.{\mathbf{B}_{#1}}				
\def\bK#1.{\mathbf{K}_{#1}}				
\def\subs#1.{_{#1}}						
\def\sups#1.{^{#1}}						

\DeclareMathOperator{\Supp}{Supp}		

\DeclareMathOperator{\Nlc}{Nlc}
\DeclareMathOperator{\Nklt}{Nklt}	
\DeclareMathOperator{\chara}{char}	

\newcommand{\rar}{\rightarrow}		
\newcommand{\drar}{\dashrightarrow}	

\usepackage{tikz}
\usetikzlibrary{matrix,arrows,decorations.pathmorphing}

  \newtheorem*{ks.princ}{
  Connectedness Principle}
  \newtheorem{theorem}{Theorem}[section]
  \newtheorem{lemma}[theorem]{Lemma}
  \newtheorem{proposition}[theorem]{Proposition}
  \newtheorem{corollary}[theorem]{Corollary}
  
  \newtheorem{conjecture}[theorem]{Conjecture}

\theoremstyle{definition}
  
  \newtheorem{definition}[theorem]{Definition}
  \newtheorem{example}[theorem]{Example}

\newtheorem{remark}[theorem]{Remark}

\theoremstyle{remark}

\numberwithin{equation}{section}

\makeatletter
\@namedef{subjclassname@2020}{
\textup{2020} Mathematics Subject Classification}
\makeatother

\begin{document}

\title{Connectedness principle for $3$-folds in characteristic $p>5$}

\author[S. Filipazzi]{Stefano Filipazzi}
\address{
EPFL, SB MATH CAG, MA C3 625 (B\^{a}timent MA), Station 8, CH-1015 Lausanne, Switzerland
}
\email{stefano.filipazzi@epfl.ch}

\author[J.~Waldron]{Joe Waldron}
\address{Department of Mathematics,
Michigan State University,
619 Red Cedar Road,
East Lansing, MI 48824, USA}
\email{waldro51@msu.edu}


\begin{abstract}
A conjecture, known as the Shokurov--Koll\'ar connectedness principle, predicts the following.
Let $(X,B)$ be a pair, and let $f \colon X \rar S$ be a contraction with $-(\K X. + B)$ nef over $S$;
then, for any point $s \in S$, the intersection $f \sups -1. (s) \cap \Nklt(X,B)$ has at most two connected components, where $\Nklt(X,B)$ denotes the non-klt locus of $(X,B)$.
This conjecture has been extensively studied in characteristic zero, and it has been recently settled in that context.
In this work, we consider this conjecture in the setup of positive characteristic algebraic geometry.
We prove this conjecture holds for threefolds in positive and mixed characteristic, where the residue fields are assumed to have characteristic $p>5$.
Under the same assumptions, we characterize the cases in which $\Nklt(X,B)$ fails to be connected.
\end{abstract}

\maketitle

\section{Introduction}

In algebraic geometry, one of the main topics of research is the study of the singularities of an algebraic variety.
One can approach the study of singularities from two complementary perspectives.
On the one hand, one can consider the germ of a singularity $x \in X$, and study its local properties.
On the other hand, one can consider a proper variety $Y$ and analyze its singularities.

From the perspective of birational geometry and the minimal model program, log canonical singularities are the broadest class of singularities that are allowed on a normal variety $X$, or, more generally, on a pair $(X,B)$.
This class of singularities admits a distinguished sub-class of singularities, called Kawamata log terminal singularities (klt for short).
This sub-class is significantly better behaved both from a geometric and cohomological perspective.
For this reason, on any given variety $X$, it is relevant to analyze the locus where $X$ fails to be klt.
More generally, given a pair $(X,B)$, it is of great importance to study the non-klt locus $\Nklt(X,B)$, that is, the locus where the pair $(X,B)$ has singularities worse than klt.

In general, the non-klt locus can be arbitrarily complicated.
On the other hand, under some natural assumptions on the positivity of the log canonical divisor $K_X + B$, it is expected that the non-klt locus of a pair $(X,B)$ has precise topological properties.
More precisely, we have the following principle, known as the Shokurov--Koll\'ar connectedness principle.

\begin{ks.princ}
Let $(X,B)$ be a pair.
Let $f \colon X \rar S$ be a contraction. 
If $-(\K X.+ B)$ is $f$-nef and $f$-big, then $\Nklt(X,B)$ is connected in a neighborhood of any fiber of $f$.
\end{ks.princ}

Recall that a morphism $f \colon X\to S$ is a contraction if it is projective and $f_*\mathcal{O}_X=\mathcal{O}_S$, which in particular implies that it has connected fibers.
This principle is expected to have extensions to the case where $-(K_X+B)$ is only nef, where $\mathrm{Nklt}(X,B)\cap f^{-1}(s)$ is expected to have at most two connected components.
In particular, the following statement has been conjectured.

\begin{conjecture} 
\label{conjecture.connectedness}
Let $(X,B)$ be a pair, and let $f \colon X \rar S$ be a contraction such that $-(\K X. + B)$ is $f$-nef.
Then, for any $s \in S$, $f \sups -1. (s) \cap \Nklt(X,B)$ has at most two connected components.
\end{conjecture}

The connectedness principle was originally formulated in characteristic zero, and it dates back to \cite{Sho}*{5.7}, where Shokurov considered the case of surface pairs $(X,B)$ with $-(K_X+B)$ ample.
Subsequently, this was generalized by Koll\'ar to higher dimensions \cite{Kol92}*{Theorem 17.4}.
Recently, the connectedness principle has been fully settled in characteristic zero, see \cites{B20,filipazzi_svaldi}.

In positive characteristic, the study of the connectedness principle is more recent and less understood.
In particular, all the known results are in dimension up to 3.
In the case of threefolds, the case where $f$ is birational and $-(K_X+B)$ is $f$-ample first appeared in work of Birkar, as part of the construction of the threefold MMP in characteristic $p>5$ in \cite{birkar_p}*{Theorem 1.8}.
Furthermore, in the case of surfaces, Birkar completely settled the case when $-(K_X+B)$ is relatively ample, see \cite{birkar_p}*{Theorem 9.3}.
These results were later extended by Nakamura and Tanaka, who fully settled the case where $-(K_X+B)$ is $f$-big and $f$-nef in dimension up to 3, see \cite{NT20}*{Theorem 1.2}.

Building on these pioneering works in positive characteristic, in this note we fully settle Conjecture \ref{conjecture.connectedness} for threefolds over perfect fields of characteristic $p > 5$, and we fully characterize the case when the non-klt locus fails to be connected.
Furthermore, in view of the recent developments in the MMP in mixed characteristic \cite{BMP+}, we are able to extend the connectedness principle to threefolds in mixed characteristic, and also partially to threefolds over imperfect fields.
In particular, our main statement is the following.

\begin{theorem}\label{thm:main-thm}
Let $R$ be an excellent domain of finite Krull dimension, which admits a dualizing complex, and all of whose closed points have characteristic zero or $p>5$.
Let $f \colon X \rar S$ be a projective morphism between varieties that are quasi-projective over $\Spec(R)$ such that $f(X)$ is not a closed point with imperfect residue field.
Assume that $\dim X \leq 3$ and that $(X,B)$ is a pair with rational coefficients and $-(\K X. + B)$ nef over $S$.
Fix $s \in S$, and assume $f^{-1}(s)$ is connected as $k(s)$-scheme.
Then, if $f \sups -1. (s) \cap \Nklt(X,B)$ is disconnected as $k(s)$-scheme, the following holds:
	\begin{itemize}
		\item[(1)] $(X,B)$ is log canonical in a neighborhood of $f^{-1} (s)$;
	    \item[(2)] $f^{-1}(s)\cap \Nklt(X,B)$ has two connected components, and, {after an \'etale base change, $\Nklt(X,B)$ has two connected components over $s$, and each of the components of $f^{-1}(s)\cap \Nklt(X,B)$ corresponds to one of the components of $\Nklt(X,B)$}; and
		\item[(3)] there is an \'etale morphism $(s' \in S') \rar (s \in S)$ and a projective morphism $T' \rar S'$ such that $k(s')=k(s)$ and the crepant pull-back of $(X,B)$ to $X \times_S S'$ is birational to a weak $\pr 1.$-link over $T'$.
	\end{itemize}
{Furthermore, if $R$ is a perfect field of characteristic $p>5$, the weak $\pr 1.$-link in (3) is a standard $\pr 1.$-link.}
\end{theorem}

\begin{remark}
	The two restrictions about perfect fields in Theorem \ref{thm:main-thm} occur for the following reasons:
	\begin{enumerate}
		\item We exclude the case of $f(X)$ being a closed point with imperfect field since the full base point free theorem and termination with scaling are not known in this situation.
		If those things are proved, the initial restriction could be removed and we would recover (1), (2), and (3) in that situation.
		\item The restriction in the ``furthermore'' statement occurs because passing from a weak $\mathbb{P}^1$-link to a standard $\mathbb{P}^1$-link requires vanishing theorems which are only proved for threefolds over perfect fields \cite{Ber20}.
		We believe that these vanishing theorems are also valid by essentially the same proofs in mixed characteristic if we assume that all residue fields are perfect, but we have not checked this.
Obtaining a standard $\mathbb{P}^1$-link from a weak $\mathbb{P}^1$-link is the only place we use vanishing theorems in this paper.
\end{enumerate}
\end{remark}

\begin{remark}
{In the context of part (2) of Theorem \ref{thm:main-thm}, the \'etale base change is necessary to have a one-to-one correspondence between the two connected components of $f^{-1}(s)\cap \Nklt(X,B)$ and the connected components of $\Nklt(X,B)$ locally over $s$.
An example, kindly suggested by the anonymous referee, is the following.
Consider the pair $(X,B)$ with $X=\mathbb{A}^1_s \times \mathbb{P}^1_{[x:y]}$ and $B=\{ sy^2=x^2\}$.
Then, over $s \neq 0$, $f^{-1}(s)\cap \Nklt(X,B)$ consists of two points, while $\Nklt(X,B)=B$ is connected.
Then, after the base change induced by $s \mapsto s^2$, which is \'etale over $s \neq 0$, the pull-back of $B$ splits as two sections of the fibration.}
\end{remark}

\begin{remark} \label{rmk_intro}
In \cite{hacon_han}, Hacon and Han reduced Conjecture \ref{conjecture.connectedness} in dimension $n$ to the termination of an arbitrary sequence of flips in dimension $n$.
Although they work in characteristic zero, their reduction also works in positive and mixed characteristic.
Furthermore, in dimensions three and four, they require only the termination from \cite{ahk}. 
After a first version of this work was completed, Hacon and Witaszek generalized many of the results of \cite{ahk} to positive characteristic, see \cite{hacon_witaszek}.
Therefore, an alternative proof of Theorem \ref{thm:main-thm} is possible by combining \cite{hacon_han} and \cite{hacon_witaszek}.
We point out that this alternative proof extends the validity of Theorem \ref{thm:main-thm} to real coefficients.

However, our proof is different from theirs, and it essentially reduces Conjecture \ref{conjecture.connectedness} to the termination of an arbitrary sequence of flips in dimension $n-1$.
 Therefore, it may be of independent interest.
 The authors believe that the argument here could be adapted to give an alternative proof of the $5$-dimensional case in characteristic zero.
 We do not pursue this in view of the recent works \cites{B20,filipazzi_svaldi}.
 Furthermore, we believe that our approach could prove itself useful to tackle the general case in positive characteristic, granted the needed developments of the MMP in that direction.
\end{remark}

The analysis of the disconnected case in Theorem \ref{thm:main-thm} has immediate applications to the study of the structure of log canonical centers on a variety.
In particular, the following corollaries follow immediately from Theorem \ref{thm:main-thm} by the same proofs as in \cite{kollar_singularities}*{\S4.4}.

\begin{corollary}\label{P1_link}
{Let $R$ be an excellent domain of finite Krull dimension, which admits a dualizing complex, and all of whose closed points have characteristic zero or $p>5$.
In case $R$ is a field of positive characteristic, we further assume it is perfect.}
Let $(X,B)$ be a dlt pair with rational coefficients with $\dim X \leq 3$ and defined over $\Spec(R)$.
Let $f \colon X \rar S$ be a projective dominant morphism such that $-(\K X. + B)$ is $f$-nef.
Fix $s \in S$ such that $f \sups -1. (s)$ is connected as $k(s)$-scheme.
Let $Z \subset X$ be minimal (with respect to inclusion) among the log canonical centers such that $s \in f (Z)$.
Let $W$ be a log canonical center such that $s \in f (W)$.
Then, there exists a log canonical center $Z_W \subset W$ such that $Z$ and $Z_W$ are {weakly} $\pr 1.$-linked and $s \in f(Z_W)$.
In particular, all the minimal (with respect to inclusion) log canonical centers $Z_i \subset X$ such that $s \in f (Z_i)$ are {weakly} $\pr 1.$-linked to each other.
{Furthermore, if $R$ is a perfect field of characteristic $p>5$, ``weakly'' can be dropped.}
\end{corollary}
For a proper morphism $f \colon X\to S$ we say that $W\subset S$ is a log canonical center for $(X,B)$ with respect to $f$ if $W=f(Z)$ for some log canonical center $Z$ of $(X,B)$.

\begin{corollary}\label{kollar_441}
Let $R$ be an excellent domain of finite Krull dimension, which admits a dualizing complex, and all of whose closed points have characteristic zero or $p>5$.
In case $R$ is a field of positive characteristic, we further assume it is perfect.
Let $f \colon X \rar S$ be a contraction  between varieties defined over $\Spec(R)$.
Assume that $\dim X \leq 3$ and that $(X,B)$ is a log canonical pair with rational coefficients and $-(\K X. + B)$ nef over $S$. Then:
\begin{itemize}
\item every point $s\in S$ is contained in a unique smallest lc center $W_s\subset S$, which is unibranch at $s$; and
\item 
any intersection of log canonical centers is also a union of log canonical centers.
\end{itemize}
\end{corollary}

\subsection*{Structure of the paper}
{
In \S\ref{sect_prelim}, we first recall the notion of generalized pair.
Since this notion is exploited in positive {and mixed} characteristic for the first time, we settle some technical results about generalized pairs for 3-folds in this setting.
In particular, we prove an adjunction result, see Lemma \ref{lemma_coeff_dcc}, which is crucial in the proof of the key statement of \S\ref{sect_prelim}, namely Proposition \ref{prop:special_termination}.
The latter is a version of special termination for generalized dlt 3-folds in positive characteristic.
This is the statement that allows our proof strategy to rely on termination in dimension one less, as mentioned in Remark \ref{rmk_intro}.
Then, in \S\ref{sec3}, we prove the main results of this work.
In \S\ref{sec3.1}, we show that the number of connected components of $\Nklt(X,B)$ is preserved by a suitable run of the MMP.
This allows us to rely on the tools from the MMP for our proof strategy.
Then, in \S\ref{sec3.2}, we prove a special case of connectedness, under the assumption that $(X,B)$ is not lc over $s \in S$.
More precisely, we show that $f\sups -1. (s) \cap \Nklt(X,B)$ is connected under these assumptions.
This is the main technical step and it relies on special termination.
Finally, in \S\ref{sec3.3}, we prove the main statements.
In Proposition \ref{prop:nklt_dominates}, we show that, if $f \sups -1. (s) \cap \Nklt(X,B)$ is disconnected, then $\Nklt(X,B)$ dominates the base $S$.
This statement is proved by reduction to the non-lc case, which had been treated in \ref{sec3.2}.
Once Proposition \ref{prop:nklt_dominates} is settled, we can prove Theorem \ref{thm:main-thm}, closely following the original approach of \cite{kollar_singularities}*{Proposition 4.37}.
Lastly, Corollary \ref{P1_link} and Corollary \ref{kollar_441} immediately follow from Theorem \ref{thm:main-thm} by the arguments in \cite{kollar_singularities}*{\S4.4}.
}

\subsection*{Acknowledgements} The authors wish to thank Christopher Hacon for feedback on an earlier version of this work, and for explaining to them the connections between \cite{hacon_witaszek} and \cite{ahk}.
They would also like to thank the referees for carefully reading the earlier versions.
{SF thanks Fabio Bernasconi for helpful conversations.}
This work was supported by a grant from the Simons Foundation (850684, JW).
	
\section{Preliminaries} \label{sect_prelim}

In this section, we collect some basic facts and definitions that will be needed in this work.

\subsection{Terminology and conventions}\label{sub:notation}

We work over an excellent domain $R$ of finite Krull dimension, which admits a dualizing complex, and all of whose closed points have characteristic zero or $p>5$.
In this work, a variety is an integral scheme, separated and of finite type over $\Spec(R)$.
All our varieties will be quasi-projective over $\Spec(R)$.

For anything not explicitly addressed in this subsection, we direct the reader to the terminology and the conventions of \cite{KM98}.

A \emph{contraction} is a projective morphism $f\colon X \rar Z$ of quasi-projective varieties with $f_* \O X. = \O Z.$.

Consider a set $\mathcal R \subset [0,1]$.
We define the \emph{set of hyperstandard multiplicities} associated with $\mathcal R$ as
\[
\Phi (\mathcal R) \coloneqq \bigg \lbrace \left. 1- \frac{r}{m}  \right| r \in \mathcal R, m \in \nn \bigg \rbrace.
\]
When $\mathcal R = \lbrace 0,1 \rbrace$, we call $\Phi(\mathcal R)$ the set of \emph{standard multiplicities}.
Usually, with no mention, we assume $0,1 \in \mathcal R$, so that $\Phi (\lbrace 0,1 \rbrace ) \subset \Phi (\mathcal R)$.
Furthermore, if $1-r \in \mathcal R$ for every $r \in \mathcal R$, we have that $\mathcal R \subset \Phi (\mathcal R)$.

Let $\mathbb{K}$ denote $\zz$, $\qq$, or $\rr$.
A \emph{$\mathbb{K}$-divisor} on a variety $X$ is a formal sum $D = \sum \subs i=1. ^n a_i P_i$, where $a_i \in \mathbb{K} \setminus \{0\}$, $n \in \nn$ and {the $P_i$ are a set of distinct prime Weil divisors on $X$ indexed by $i=1, \ldots, n$.} 
We say that $D$ is $\mathbb{K}$-Cartier if it can be written as a $\mathbb{K}$-linear combination of $\zz$-divisors that are Cartier.
Given a $\mathbb{K}$-divisor $D=\sum_{i=1}^n a_iP_i$, we define its \textit{support} to be the union of the prime divisors appearing with non-zero coefficient in its decomposition, i.e., $\mathrm{Supp}(D)= \cup_{i=1}^n P_i$.

In all of the above, if $\mathbb{K}= \zz$, we will drop it from the notation.

Given a $\mathbb{K}$-divisor $D$ and a prime divisor $P$ in its support, we denote by $\mu_P (D)$ the coefficient of $P$ in $D$.
Given a divisor $D = \sum \subs i=1. ^n \mu_{P_i}(D) P_i$, we define its {\it round down} as $\lfloor D \rfloor  \coloneqq \sum \subs i=1.^n \lfloor \mu_{P_i}(D) \rfloor P_i$.
Similarly, we define its {\it round up} as $\lceil D \rceil  \coloneqq \sum \subs i=1.^n \lceil \mu_{P_i}(D) \rceil P_i$.
For a $\mathbb{K}$-divisor $D=\sum_{i=1}^n a_i P_i$ and $c \in \rr$, we set $D \sups \geq c. \coloneqq \sum_{a_i \geq c}a_i P_i$, that is, we ignore the components with coefficient strictly less than $c$.
Similarly, we define $D \sups \leq c.$, $D \sups > c.$,  $D \sups < c.$, and $D \sups = c.$.

Given a $\mathbb{K}-$divisor $D = \sum a_i P_i$ on a normal variety $X$, and a morphism $\pi \colon X \to Z$, we define
\[
D^v \coloneqq \sum_{\pi(P_i) \subsetneqq Z} a_i P_i, \
D^h \coloneqq \sum_{\pi(P_i) = Z} a_i P_i.
\]
We call $D^v$ the $\pi$-vertical part of $D$ and $D^h$ the $\pi$-horizontal part of $D$, respectively.

\subsection{B-divisors}
\label{b-div.subs}

Let $\mathbb{K}$ denote $\zz$, $\qq$, or $\rr$.
Given a normal variety $X$, a {\it $\mathbb{K}$-b-divisor} $\mathbf{D}$ is a (possibly infinite) sum of geometric valuations $V_i$ of $k(X)$ with coefficients in $\mathbb{K}$,
\[
 \mathbf{D}= \sum_{i \in I} a_i V_i, \; a_i \in \mathbb{K},
\]
such that for every normal variety $X'$ birational to $X$, only a finite number of the valuations $V_i$ can be realized by divisors on $X'$. 
We define the {\it trace} $\mathbf{D}_{X'}$ of $\mathbf{D}$ on $X'$ as 
\[
\mathbf{D}_{X'} \coloneqq \sum_{
\{i \in I \; | \; c_{X'}(V_i)= D_i, \; 
\codim_{X'} D_i=1\}} a_i D_i,
\]
where $c_{X'}(V_i)$ denotes the center of the valuation on $X'$.

A $\mathbb{K}$-b-divisor $\mathbf{D}$ over $X$ is said {\it $\mathbb{K}$-b-Cartier} if the following holds: there exists a birational model $X'$ of $X$ such that $\mathbf{D}_{X'}$ is $\mathbb K$-Cartier on $X'$ and for any model $r \colon X''  \rar X', \; \mathbf{D}_{X''} = r^\ast \mathbf{D}_{X'}$.
When this is the case, we say that $\mathbf{D}$ descends to $X'$ and write $\mathbf{D}= \overline{\mathbf{D}_{X'}}$.
Given a birational model $\pi \colon X' \rightarrow X$ and a $\mathbb{K}$-Cartier divisor $D'$ on $X'$, we induce a $\mathbb{K}$-b-Cartier divisor $\overline{D'}$ on $X$ as follows:
for any other model $X''$ of $X$, we take a model $\tilde X$ resolving $X'' \dashrightarrow X'$, and we define $\overline{D'}_{X''}\coloneqq p_*q^*D'$, where we have $p \colon \tilde X \rightarrow X''$ and $q \colon \tilde X \rightarrow X'$.
Notice that, in this case, we have $\overline{D'}_{X'}=D'$ and $\overline{D'}_X=\pi_*D'$.
The $\mathbb{K}$-b-divisor $\overline{D'}$ is called {\it $\mathbb{K}$-b-Cartier closure} of $D'$.
We say that a $\mathbb{K}$-b-divisor $\mathbf{D}$ is {\it b-nef}, if it is $\mathbb{K}$-b-Cartier and there exists a model $X'$ of $X$ such that $\mathbf{D}= \overline{\mathbf{D}_{X'}}$, where $\mathbf{D}_{X'}$ is nef on $X'$. 
The notion of b-nef b-divisor can be extended analogously to the relative case.

In all of the above, if $\mathbb{K}= \zz$, we will drop it from the notation.

\subsection{Generalized pairs and singularities}
\label{sect.gen.pairs.sings}

We recall the definition of generalized pairs, first introduced in~\cite{BZ16}, which is a generalization of the classic setting of log pairs.

\begin{definition}
A {\em generalized sub-pair} $(X,B, \mathbf{M})/Z$ over $Z$ is the datum of:
\begin{itemize}
\item a normal variety $X$ together with a projective morphism $X \rar Z$;
\item an $\mathbb R$-Weil divisor $B$ on $X$; and
\item an $\mathbb{R}$-$b$-Cartier $\mathbb R$-b-divisor $\mathbf{M}$ over $X$ that is b-nef over $Z$.
\end{itemize}
Moreover, we require that $K_X +B+ \mathbf{M}_X$ is $\mathbb R$-Cartier.
If $B$ is effective, we say that $(X,B,\bM.)/Z$ is a generalized pair.
If $\bM . =0$, we drop the word ``generalized''.
\end{definition}

When the setup is clear, we will denote the generalized sub-pair $(X, B, \mathbf{M})/Z$ by $(X,B+\mathbf{M}_X)/Z$ and we will say that $(X,B+\mathbf{M}_X)$ is a generalized pair over $Z$ with datum $\mathbf M$; for simplicity, we will often replace $\mathbf{M}_X$ by $M$, and write $(X, B+M)$.
When $Z$ is the spectrum of the ground field, we will simply write $(X, B, \mathbf{M})$ and $(X, B +M)$.

Let $(X,B, \mathbf{M})/Z$ be a generalized sub-pair and let $\rho\colon X' \rar X$ be a projective birational morphism. 
Then, we may write
\[
\K X'.+B_{X'} + \mathbf{M}_{X'}=\rho^\ast (K_X+B+M),
\]
where we choose $K_{X'}$ to satisfy $\rho_*(K_{X'})=K_X$.
Given a prime divisor $E$ on $X'$, we define the {\em generalized log discrepancy} of $E$ with respect to $(X,B,\bM.)/Z$ to be $a_E(X,B,\bM.)\coloneqq 1-\mu_{E}(B_{X'})$.

\begin{definition}
Let $(X,B, \mathbf{M})/Z$ be a generalized sub-pair.
We say that $(X,B,\bM.)/Z$ is \emph{generalized sub-log canonical} if $a_E(X,B,\bM.) \geq 0$ for all divisors $E$ over $X$.
Similarly, if $a_E(X,B,\bM.) > 0$ for all divisors $E$ over $X$ and $\lfloor B \rfloor \leq 0$, we say that $(X,B,\bM.)/Z$ is \emph{generalized sub-klt}.
When $B \geq 0$, we say that $(X,B,\bM.)/Z$ is \emph{generalized log canonical} or \emph{generalized klt}, respectively.
\end{definition}

\begin{definition}
\label{lc.center.def}
Let $(X,B,\bM.)/Z$ be a generalized sub-pair and let $E$ be a divisor over $X$.
We say that $E$ is a {\em non-klt place} for the generalized pair if $a_E(X,B,\bM.) \leq 0$; in this case, we say that $c_X(E)\subset X$ is a {\em non-klt center} for the generalized pair. 
The \emph{non-klt locus} $\Nklt(X,B, \mathbf M)$ is defined as the union of all the non-klt centers of $(X,B,\bM.)/Z$. 
If $a_E(X,B,\bM.) = 0$, we say that $E$ is a {\em generalized log canonical place} for $(X,B,\bM.)/Z$, and, under the assumption that $(X,B,\bM.)/Z$ is generalized sub-log canonical in a neighborhood of $c_X(E)$, $c_X(E)$ is a {\em generalized log canonical center} for $(X,B,\bM.)/Z$.
Similarly, we define the \emph{non-log canonical locus} $\mathrm{Nlc}(X,B,\bM.)$ as the union of all the generalized non-klt centers of $(X,B,\bM.)/Z$ that are not generalized log canonical centers.
\end{definition}

The classical results on adjunction for pairs, cf.~\cite{birkar_p}*{\S4}, can be extended to the context of generalized pairs.
Furthermore, a slightly weaker version of log resolutions is known to exist in our setup by \cites{BMP+,kollar_witaszek}.  The version we will need is the following.

\begin{proposition}[{\cite{BMP+}*{Theorem 2.12},\cite{kollar_witaszek}*{Theorem 1}}]\label{prop:resolution}
	Let $X$ be a reduced scheme of dimension $3$, quasi-projective over a quasi-excellent affine scheme $\Spec(R)$.  Let $T$ be a subscheme of $X$.
	Then there exists a projective birational morphism $g \colon Y\to X$ from a regular scheme $Y$ such that both $g^{-1}(T)$, and $\mathrm{Ex}(g)$ are divisors, $\Supp(g^{-1}(T)\cup\mathrm{Ex}(g))$ is simple normal crossing and $Y$ supports a $g$-ample $g$-exceptional divisor.
\end{proposition}

We will refer to such a $g$ as a log resolution of $(X,T)$, but note that it does not require $g$ to be an isomorphism over the regular locus of $X$.
Given a pair $(X,B)$, by definition, a log resolution of $(X,B)$ is a log resolution of $(X,\Supp(B))$.

Let $(X,B,\mathbf{M})/Z$ be a generalized pair.
Let $S$ be an irreducible component of $\lfloor B \rfloor$, and let $S^\nu$ be its normalization.
Let $f \colon X' \rar X$ be a log resolution of $(X,B)$ where $\mathbf{M}$ descends.
Let $S'$ be the strict transform of $S$ on $X'$, and
denote by $g  \colon  S' \rar S^\nu$ the induced morphism.
Then, we can write
\[
\K X'. + B' + \mathbf{M} \subs X'. = f^\ast  (\K X. + B + \mathbf{M}_X),
\]
where we assume $f_*(\K X'.)=K_X$.
Up to replacing $\mathbf{M}_{X'}$ in its $\mathbb{K}$-linear equivalence class, we may assume that $S'$ does not appear in the support of $\mathbf{M} \subs X'.$.
Then, we define
\[
\K S'. + B_{S'} + \mathbf{N}_{S'} \coloneqq (\K X'. + B' + \mathbf{M} \subs X'.)|_{S'},
\]
where we have $B_{S'} \coloneqq (B'-S')|_{S'}$, and $\mathbf{N} \subs S'. \coloneqq \mathbf{M} \subs X'.|_{S'}$. 
Then, we set $B_{S^\nu} \coloneqq g_* B_{S'}$, and $\mathbf{N}_{S^\nu} \coloneqq g_* \mathbf{N}_{S'}$.
By construction, it follows that
\[
\K S^\nu. + B_{S^\nu} + \mathbf N _{S^\nu} = (\K X. + B + \mathbf{M}_X)|_{S^\nu}.
\]
We refer to this operation as {\it generalized divisorial adjunction}.
By construction, $(S^\nu,B \subs S^\nu.,\mathbf{N})/Z$ is a generalized pair over $Z$.
If appropriate, we may write $\mathbf M | \subs S^\nu.$ for $\mathbf N$, in order to highlight that $\mathbf N$ comes from the restriction of $\mathbf M$ to $S$.
Notice that, if $p \bM.$ is b-Cartier for some $p \in \nn$, then so is $p \mathbf N$.

\begin{remark} \label{rmk_boundary_adj}
If $(X,B,\bM.)/Z$ is generalized log canonical, the divisor $B \subs S^\nu.$ defined by divisorial adjunction is a boundary, and $(S^\nu,B_{S^\nu},\mathbf N)/Z$ is generalized log canonical.
Indeed, by the construction involving the log resolution, it is immediate that every generalized log discrepancy of $(S^\nu,B_{S^\nu},\mathbf N)/Z$ is non-negative.
Thus, one just needs to argue that $B_{S^\nu}$ is effective.
This follows by arguing as in \cite{BZ16}*{Remark 4.8}.
\end{remark}

\begin{lemma} \label{lemma_coeff_dcc}
Let $p \in \nn$ and $\mathcal{R} \subset [0,1]$ be a finite set of rational numbers.
Then, there exists a finite set of rational numbers $\mathcal{S} \subset [0,1]$, only depending on $p$ and $\mathcal{R}$, satisfying the following.
Assume that
\begin{itemize}
\item $(X,B,\bM.)/Z$ is a generalized log canonical pair of dimension at most 3;
\item $S^\nu$ is the normalization of a component of $\lfloor B \rfloor$;
\item $\mathrm{coeff}(B) \subset \Phi(\mathcal R )$ and $p \bM.$ is b-Cartier; and
\item $(S^\nu,B \subs S^\nu.,\mathbf N)/Z$ is the generalized pair obtained by generalized divisorial adjunction.
\end{itemize}
Then, $\mathrm{coeff}(B \subs S^\nu.) \subset \Phi(\mathcal{S})$.
\end{lemma}

\begin{proof}
We follow the proofs of \cite{birkar_p}*{Proposition 4.2} and \cite{birkar_complements}*{Lemma 3.3}.
As in the proof of \cite{birkar_p}*{Proposition 4.2}, by localization, we reduce to the case when $X$ is a normal excellent scheme of dimension 2, and we work around a fixed closed point $V$.
Thus, $(X,B)$ is a log canonical pair.
Furthermore, as in the proof of \cite{birkar_p}*{Proposition 4.2}, we may assume that $S$ is regular and that $S^\nu \rar S$ is an isomorphism.

Let $\tilde B \subs S.$ be the boundary divisor obtained by $(\K X. + B)|_S= \K S. + \tilde B _S$.
By \cite{kollar_singularities}*{3.35.(1)} and \cite{birkar_p}*{Proposition 4.2}, there is a positive integer $l \in \nn$ such that the Cartier index of every integral Weil divisor near $V$ is divides $l$ and
\begin{equation} \label{eq_coeff}
\mu_V(\tilde B _S)= 1- \frac{1}{l} + \sum \subs i. \frac{b_i \alpha_i}{l},
\end{equation}
where $\alpha_i \in \nn$ and $B= \sum_i b_i P_i$ and each $P_i$ is a prime divisor.
Now, \eqref{eq_coeff} is the input to apply the proof \cite{birkar_complements}*{Lemma 3.3} with no changes.
This concludes the proof.
\end{proof}

\subsection{Generalized dlt pairs and generalized MMP}
\label{gen.dlt.mod.ssect}
In this section, we recall the notion of dlt and plt singularities, and we extend it to the context of generalized pairs.
Furthermore, we prove some technical statements regarding the minimal model program (MMP) for these objects.

\begin{definition}
\label{gen.dlt.def}
A generalized pair $(X,B, \mathbf{M})/Z$ is called {\em generalized dlt}, if it is generalized log canonical and for the generic point $\eta$ of any generalized log canonical center the following conditions hold:
\begin{itemize}
    \item[(i)] $\bM . = \overline{\bM X.}$ over a neighborhood of $\eta$; and
    \item[(ii)] $(X,B)$ is log smooth in a neighborhood of $\eta$.
\end{itemize}
If, in addition, every connected component of $\lfloor B \rfloor$ is irreducible, we say that $(X,B,\bM.)/Z$ is {\em generalized plt}.
\end{definition}

\begin{remark} \label{rmk.gen.dlt}
By Definition \ref{gen.dlt.def}, it follows that there is an open subset $U \subset X$ containing all the generic points of the generalized log canonical centers of $(X,B,\bM.)/Z$ such that $U$ is smooth and the generalized log canonical places of $(X,B,\bM.)/Z$ coincide with the log canonical places of $(U,B|_U)$.
In particular, if $(X,B,\bM.)/Z$ is generalized plt, all the divisors $E$ that are exceptional over $X$ satisfy $a_E(X,B,\bM.)>0$.
\end{remark}

\begin{remark} \label{dlt_adj}
{
Let $(X,B, \mathbf{M})/Z$ be a generalized dlt pair of dimension at most 3, and let $Z$ be a stratum of $\lfloor B \rfloor$.
Then, we can iterate divisorial adjunction on the higher dimensional strata containing $Z$ to induce a generalized pair structure on $Z$.
It is easy to see that this process does not depend on the order we consider the strata containing $Z$.
Indeed, by induction on the number of strata and the dimension, it suffices to consider the case of two prime components $T,S \subset \lfloor B \rfloor$ with $Z = S \cap T$.
Then, by the generalized dlt property and the assumptions on the dimension, we can find a log resolution $\pi \colon X' \rar X$ of $(X,B)$ where $\bM.$ descends and such that $\pi$ is an isomorphism at the generic point of $S \cap T$.
Then, the strict transforms $S'$ and $T'$ of $S$ and $T$ intersect along the strict transform $Z'$ of $Z$, and by direct computation, we can see that doing adjunction first on $S'$ and then on $S' \cap T'$ or first on $T'$ and then on $S' \cap T'$ lead to the same outcome on $Z'$.}
\end{remark}

\begin{definition}
We say generalized pairs $(X_1,B_1, \bM 1.)$ and $(X_2,B_2, \bM 2.)$ are {\it crepant} if $X_1$ and $X_2$ are birational, $\bM 1. = \bM 2. \eqqcolon \bM.$, and there is a common resolution $\phi_i \colon X \to X_i$ such that
\[
\phi_1^*(K_{X_1}+B_1+\bM X_1.)=\phi^*_2(K_{X_2}+B_2+ \bM X_2.).
\]
\end{definition}

\begin{lemma}\label{lem:plt_crepant}
Let $(X_1,B_1,\bM.)/Z$ and $(X_2,B_2,\bM.)/Z$ be two crepant $\bQ$-factorial generalized pairs, such that there is a birational map $\phi \colon X_1\dashrightarrow X_2$ that does not contract any component of $\rddown{B_1}$.
Then, if $(X_1,B_1, \bM.)$ is generalized plt, so is $(X_2,B_2, \bM.)$.
\end{lemma}

\begin{proof}
By Remark \ref{rmk.gen.dlt}, this follows immediately from the usual version for pairs \cite{kollar_singularities}*{Corollary 4.35}.
\end{proof}

\begin{lemma}[cf. {\cite{BMP+}*{Lemma 2.34}}]\label{lemma add ample}
Let $(X,B)$ be a $\qq$-factorial klt pair with $\dim X \leq 3$ over $\Spec(R)$. 
Let $f \colon X\to Z$ be a projective contraction, and let $A$ be a relatively ample $\mathbb R$-divisor.
Then around any $z\in Z$, there is some open subset $z\in U\subset Z$ such that we may find $0\leq A'\sim_{\rr,U} A|_{X_U}$ such that $(X_U,B_U+A')$ is klt.
\end{lemma}

\begin{proof}
First we prove the corresponding statement over the local ring $\mathcal{O}_{Z,z}$, by reducing to the case of $\mathbb{Q}$-divisors $A$ and $B$, which was dealt with in \cite{BMP+}*{Lemma 2.34}.
Let $Y=X\times_Z\Spec(\mathcal{O}_{Z,z})$.
Let $E$ be an effective $\rr$-divisor such that $A_Y- E_Y$ is an $f_Y$-ample $\mathbb{Q}$-divisor, where $E_Y$ is sufficiently small that $(Y,B_Y+ E_Y)$ is klt.
Now let $\tilde{B}_Y\geq B_Y+\epsilon E_Y$ be a $\mathbb{Q}$-divisor such that $(Y,\tilde{B}_Y)$ is klt.
Then, apply \cite{BMP+}*{Lemma 2.34} to obtain $A'_Y\sim_{\mathbb{Q}}A_Y- E_Y$ such that $(Y,\tilde{B}_Y+A'_Y)$ is klt.
Finally, $A'_Y+E_Y$ solves the original problem for $Y\to \mathcal{O}_{Z,z}$. 
But now there is some open subset $U$ containing $z\in Z$ such that $A'_Y$ spreads out to a divisor $A'$, satisfying $0\leq A'\sim_{\rr, U} A|_{X_U}$ and $(X_U,B_U+A')$ is klt.
\end{proof}

\begin{lemma}\label{local_termination}
Let $X$ be a $\mathbb{Q}$-factorial variety over $\Spec(R)$, with $\dim(X)\leq 3$.
Let $f \colon X\to Z$ be a projective contraction such that the image of $f$ is positive dimensional, and let $A$ be an $f$-ample $\mathbb{R}$-divisor.
Suppose that $D$ is an $\mathbb{R}$-Cartier divisor on $X$ such that $D+A$ is $f$-nef and there is a finite open cover $\{U_j\}_{j=1}^n$ of $Z$ such that for each $j$ there is boundary $B_{U_j}$ on $X_{U_j}$ such that $(X_{U_j}, B_{X_{U_j}})$ is klt and $D_{X_{U_j}}\sim_{\mathbb{R}, U_j}K_{X_{U_j}}+B_{X_{U_j}}$.
	
Then we can run a $D$-MMP over $Z$ with scaling of $A$ and it terminates.
\end{lemma}

\begin{proof}
We first show that the cone theorem holds for $D$, following the proof in \cite{BMP+}*{Theorem 9.28}.
The first step of the proof of the cone theorem \cite{BMP+}*{Theorem 9.28} is to find a curve in any fixed $D$-negative extremal ray $\Gamma$ of $\overline{NE}(X/Z)$ satisfying the intersection bound with $D$.
Note that for any open subset $U\subset Z$ we have a linear map $\overline{NE}(X_U/U)\xrightarrow{\iota_*} \overline{NE}(X/Z)$ given by push-forward of cycles, where $\iota$ is the inclusion $\iota:X_U\to X$.
We claim that this is well defined and injective.
Note first that for every Cartier divisor $N$ on $U$, some multiple of $\overline{N}\subset X$ is Cartier, and for any $1$-cycle $C$ on $X_U$ which is vertical over $U$, $C\cdot N=\iota_*C\cdot \overline{N}$ since $\iota$ is an isomorphism near the support of $C$.
Conversely, for any Cartier divisor $M$ on $X$, $\iota^*M$ is a well defined Cartier divisor on $X_U$ and $C\cdot \iota^*M=\iota_*C\cdot M$ since again $\iota$ is an isomorphism near the support of $C$.
In other words, the projection formula holds for the non-proper $\iota \colon X_U\to X$ so long as we only look at $1$-cycles over $Z$. 
Now 
\[[\iota_*C]=0 \iff \iota_*C\cdot N=0\ \forall N\subset X\iff C\cdot M=0\ \forall M\subset X_U\iff [C]=0\]
verifies the injectivity and well definedness of $\iota_*$ on $\overline{NE}(X_U/U)$.

We now claim that the extremal ray $\Gamma$ of $\overline{NE}(X/Z)$ is the image of an extremal ray $\Gamma_j$ of $\overline{NE}(X_{U_j}/U_j)$ for some $j=1,\ldots,n$.
For any non-zero class $\alpha\in \Gamma$ we must have $k=\alpha\cdot A>0$ for our fixed ample Cartier divisor $A$, and if we let $\{C_i\}_{i \in \mathbb{N}}$ be effective 1-cycles such that $[C_i]\to \alpha$ as $i \to \infty$, then by truncating the sequence we may assume that $3k/2>C_i\cdot A>k/2$ for all $i$.
Note that since the effective $1$-cycle $C_i$ is vertical over $Z$, there is a well defined $1$-cycle $C_i\cap X_{U_j}$ on $X_{U_j}$ over $U$ (however this does not descend to a well defined map modulo numerical equivalence).
Furthermore, for each $i$ we must have
$$n\left(C_i\cdot A\right)\geq \sum_{j=1}^n(C_i\cap X_{U_j})\cdot A_{X_{U_j}}\geq C_i\cdot A
$$
by the projection formula, since every irreducible component of $C_i$ is contained in at least one $X_{U_j}$ and hence is the push-forward of a cycle on $X_{U_j}$, which gives the right hand inequality.
However each component of $C_i$ is counted at most $n$ times in the central term and hence we obtain the left hand inequality.
Therefore we must have that for each $i$, there is $j$ such that  $$n3k/2\geq (C_i\cap X_{U_j})\cdot A|_{X_{U_j}}>k/2n,$$ where we note that the left hand inequality does not depend on the choice of $j$.
Since the cover is finite, by taking a subsequence we may assume that both inequalities hold for  some fixed $j$ and all $i$.
For this choice of $j$, the sequence $[C_i\cap X_{U_j}]\subset \overline{NE}(X_{U_j}/U_j)$ has a convergent subsequence, since every class is contained in
$$\{\beta\in \overline{NE}(X_{U_j}/U_j)\mid \beta\cdot A\leq n3k/2\},$$
which  is a compact subset.
Replace the sequence by this subsequence and let $\alpha_U=\lim_i[C_i\cap X_{{U_j}}]$.
Notice that $\alpha_U \neq 0$, since $(C_i \cap X_{U_j}) \cdot A|_{X_{U_j}} > k/2n$ for all $i$.

We claim that $\alpha_U$ generates an extremal ray $\Gamma_U$ in $\overline{NE}(X_{U_j}/U_j)$ such that $\iota_*\Gamma_U=\Gamma$.
Firstly it is clear that for each $i$, we can express $C_i$ as a sum of effective cycles as:
$$
C_i=\iota_*(C_i\cap X_{U_j})+(C_i-\iota_*(C_i\cap X_{U_j})).
$$
Therefore, we have
$$
\alpha=\lim_i [C_i]=\lim_{i}[\iota_*(C_i\cap X_{U_j})]+\lim_{i}[C_i-\iota_*(C_i\cap X_{U_j})].
$$
We have already shown that $\lim_{i}[\iota_*(C_i\cap X_{U_j})]$ exists as a non-zero pseudo-effective cycle, and so the second limit also exists and hence we can split $\lim_i [C_i]$ as a sum of limits.
Since $C_i-\iota_*(C_i\cap X_{U_j})$ is effective for every $i$, it follows that $\lim_{i}[C_i-\iota_*(C_i\cap X_{U_j})]$ is a pseudo-effective cycle, possibly 0.
Now since $\alpha$ is in the extremal ray $\Gamma$ and $\lim_{i}[\iota_*(C_i\cap X_{U_j})] \neq 0$, the definition of extremal ray implies that $\lim_{i} [\iota_*(C_i\cap X_{U_j})]$ is also a generator of $\Gamma$.
It remains to show that $\lim_{i}[C_i\cap X_{U_j}]$ generates an extremal ray in $\overline{NE}(X_{U_j}/U_j)$.  We have already shown that it generates a non-zero ray, and the fact that it is extremal follows from the injectivity of $\overline{NE}(X_U/U)\xrightarrow{\iota_*} \overline{NE}(X/Z)$ and the definition of extremal ray.

Now we can find the required curve over $U_j$ by applying \cite{BMP+}*{Theorem 9.28} to the pair $(X_{U_j}, B_{U_j})$ and the extremal ray $\Gamma_{U}$.
The non-accumulation and countability of extremal rays is a formal consequence of the existence of the bound exactly as in the proof of \cite{BMP+}*{Theorem 9.28}.
	
Now we run the $D$-MMP, for which we must justify that the steps exist.  Given a $D$-negative extremal ray in $\overline{NE}(X/Z)$, it follows formally from the cone theorem that there is a big and nef divisor $A$ on $X$ which vanishes exactly on curves in $\Gamma$. 
To contract the extramal ray, it is enough to show that $A$ is semiample, which can be checked locally on each $U_j$. 
That this holds follows from \cite{BMP+}*{Theorem 9.33}.
Similarly, the existence of flips can be checked locally, and this follows from \cite{BMP+}*{Theorem 9.14}.
Finally, the LMMP terminates since it terminates over each of the finitely many open subsets $U_j$ by \cite{BMP+}*{Theorem 9.37}.
\end{proof}


\begin{remark}
Lemma~\ref{local_termination} was formulated for termination with scaling, but note however that it could equally well have been formulated for any other kind of termination which holds over the open subsets.
\end{remark}

\begin{proposition}\label{prop:non-pseff-termination-2}
Let $(X,B)$ be a $\qq$-factorial dlt pair with $\dim X \leq 3$ over $\Spec(R)$.
Let $f \colon X\to Z$ be a projective contraction such that one of the following holds:
\begin{enumerate}
	\item the image of $X$ in $Z$ has positive dimension; or
	\item the image of $X$ in $Z$ is a point with perfect residue field of characteristic $p>5$.
	\end{enumerate}
	 and let $M$ be an $f$-nef divisor.
Assume that $K_X+B+M$ is not pseudo-effective over $Z$.
Then, there is a $(K_X+B+M)$-MMP over $Z$ which terminates.
\end{proposition}

\begin{proof}
First we deal with case (1).
Let $A$ be an $f$-ample divisor on $X$ such that $K_X+B+M+A$ is $f$-nef.
We claim that we can run a terminating $(K_X+B+M)$-MMP over $Z$ with scaling of $A$.
The assumption that $K_X+B+M$ is not pseudo-effective over $Z$ implies that if $0 < \epsilon \ll 1$, the required MMP is the same as an MMP for $K_X+B+M+\epsilon A$ with scaling of $(1-\epsilon)A$.
Now, fix $0 <\delta \ll \epsilon$ so that $\epsilon A+\delta B \sups =1.$ is $f$-ample.
As $M$ is $f$-nef, then $\epsilon A + \delta B \sups =1. + M$ is $f$-ample.

By Lemma \ref{lemma add ample} and noetherianity of $Z$, we may find a finite open cover $\{U_j\}_{j=1}^n$ of $Z$ such that for each $j$ there is
	$0 \leq A_j \sim_{\rr,Z} \epsilon A_{X_{U_j}} + \delta B_{X_{U_{j}}} \sups =1. + M_{X_{U_j}}$ such that $(X_{U_j},B_{U_j} \sups <1. + (1-\delta)B_{U_j} \sups =1. + A_j)$ is klt.
Thus, the MMP exists and terminates  by Lemma \ref{local_termination}.

In case (2) the MMP exists and terminates by \cite{GNT}*{Theorem 2.13}.
\end{proof}

\begin{proposition}[Special termination]\label{prop:special_termination}
Let $(X,B,\bM.)/Z$ be a $\qq$-factorial generalized dlt pair of dimension $3$ 
over $\Spec(R)$.
Assume that $B$ is a $\qq$-divisor and $\bM.$ $\qq$-b-Cartier.
Then, for any sequence of $(K_X+B+M)$-flips over $Z$, after finitely many flips the flipping locus does not intersect the strict transform of $\Supp(\rddown{B})$.
\end{proposition}

\begin{proof}
We will follow the proof of \cite{birkar_p}*{Proposition 5.5}.
By \cite{birkar_complements}*{2.13.(2)}, the generalized dlt property is preserved under the MMP.
Since no generalized log canonical center can be created by the MMP, up to truncating the MMP, we may assume that no generalized log canonical center is contracted by the MMP.
In particular, the MMP is an isomorphism near the generalized log canonical centers of dimension 0.

Let $C$ be a generalized log canonical center of dimension 1.
As $(X,B+M)$ is generalized dlt, there are two irreducible components $S$ and $T$ of $\lfloor B \rfloor$ such that $C$ is a component of $S \cap T$.
Let $(X_i,B_i+M_i)$ denote the $i$-th step of the MMP, and let $C_i,S_i \subset X_i$ denote the birational transforms of $C$ and $S$, respectively.

Since $\Nklt(X_i,B_i,\mathbf{M})=\Nklt(X_i,B_i)$ where $(X_i,B_i)$ is a $\mathbb{Q}$-factorial dlt pair, the pair $(X,S_i)$ is plt and so $S_i$ is normal by \cite{BMP+}*{Corollary 7.15}.
Hence the pair $(S_i, \mathrm{Diff}_{S_i}(B_i))$ is also dlt and hence by a similar argument, each irreducible component of $S_i\cap T_i$ is normal and so $C_i$ is normal.
By generalized divisorial adjunction, we can write $(\K X_i. + B_i + M_i)| \subs S_i. = \K S_i. + B \subs S_i. + M \subs S_i.$ and $(\K S_i. + B \subs S_i. + M \subs S_i.)| \subs C_i. = \K C_i. + B \subs C_i. + M \subs C_i.$.

By assumption, no generalized log canonical center is contracted, thus we have $C_i \simeq C \subs i+1.$ and we can therefore compare divisors defined on $C_i$ and $C \subs i+1.$, respectively.
By iterated generalized adjunction and the fact that $C$ is a curve, we have $M \subs C_i. = M \subs C_{i+1}.$.
Indeed, let $\tilde X$ be a common resolution of $X_i$ and $X \subs i+1.$ where $\bM.$ descends.
Let $\tilde M$ denote the trace of $\bM.$ on $\tilde X$, and let $\tilde S$ denote the strict transform of $S_i$.
We may assume that $\tilde S$ is smooth.
Let $(\tilde X,\tilde B _i,\bM.)$ and $(\tilde X,\tilde B _{i+1},\bM.)$ denote {the crepant pull-backs} of $(X_i,B_i+M_i)$ and $(X_{i+1},B_{i+1}+M_{i+1})$ on $\tilde X$, respectively.
Then, the generalized pair $(S_i,B \subs S_i. + M \subs S_i.)$ (resp. $(S_{i+1},B \subs S_{i+1}. + M \subs S_{i+1}.)$) is determined by generalized divisorial adjunction of $(\tilde X,\tilde B _i,\bM.)$ (resp. $(\tilde X,\tilde B _{i+1},\bM.)$).
In particular, the moduli part is independent of $j \in \{i,i+1 \}$.
Then, we may apply this argument again in having $S_i$ and $S_{i+1}$ as ambient spaces and performing generalized divisorial adjunction to $C_i$ and $C_{i+1}$.
Again, the moduli part on $C_j$ is independent of $j \in \{i,i+1\}$.
As $C_i$ and $C_{i+1}$ are curves, they do not have any higher model, so it follows that $M \subs C_i. = M \subs C_{i+1}.$ under the identification $C_i = C \subs i+1.$.
Furthermore, as the MMP is $(\K X. + B + M)$-negative, the generalized log discrepancies of $(X_{i+1},B_{i+1}+M_{i+1})$ are greater than or equal to the corresponding generalized log discrepancies of $(X_i,B_i+M_i)$. In particular, by repeated generalized divisorial adjunction, we have $B_{C_i} \geq B_{C_{i+1}}$.
By Lemma \ref{lemma_coeff_dcc}, the coefficients of $B \subs S_i.$ and $B \subs C_i.$ are in a fixed DCC set.
Therefore, we have $B \subs C_i. = B \subs C_{i+1}.$ for $i \gg 0$.
In particular, the MMP is an isomorphism around $C_i$.

By the previous argument, for $i \gg 0$ the MMP is an isomorphism around the generalized log canonical centers of dimension at most 1.
Up to {truncating the MMP}, we may assume this is true for the whole MMP.
Now, let $S$ be an irreducible component of $\lfloor B \rfloor$, {and let $S_i$ denote its strict transform on $X_i$}.
For each step of the MMP, let $X_i \rar W_i$ be the birational contraction leading to the flip $X_i \drar X \subs i+1.$.
{To conclude, we need to show that, for $i \gg 0$, the exceptional locus of $X_i \rar W_i$ is disjoint from every irreducible component of $\lfloor B_i \rfloor$.
First, we will show that, for every $i$, the exceptional locus of $X_i \rar W_i$ cannot intersect two distinct components of $\lfloor B_i \rfloor$.
Then, we will argue that, for every $S$, the exceptional locus of $X_i \rar W_i$ can intersect $S_i$ only finitely many times.}

{Now, let $E_i$ denote the exceptional locus of $X_i \rar W_i$, and assume that two distinct components $S_i$ and $T_i$ of $\lfloor B_i \rfloor$ intersect $E_i$.}
If $S_i \cap T_i \cap E_i \neq \emptyset$, we get a contradiction, as the MMP is an isomorphism {around} $S_i \cap T_i$.
Thus, neither $S_i$ nor $T_i$ contain $E_i$, and both of them intersect $E_i$ at {distinct} isolated points.
{Since neither $S_i$ nor $T_i$ contains any component of $E_i$, both of them intersect the curve $E_i$, and $X_i \rar W_i$ is an extremal contraction, it follows that $S_i$ and $T_i$ are relatively ample over $W_i$.
Thus, by the definition of flip, $S_{i+1}$ and $T_{i+1}$ is relatively anti-ample over $W_i$.
In particular, both $S_{i+1}$ and $T_{i+1}$ contain the exceptional curves for $X_{i+1} \rar W_i$.
Then, it follows that $S_{i+1} \cap T_{i+1}$ contains a curve that was not present in $S_i \cap T_i$.
Hence, a new generalized log canonical center is created, contradicting the negativity of the MMP.}
This leads to the sought contradiction.

{Now, we are left with showing that, for a given component $S$ of $\lfloor B \rfloor$, its strict transform intersects the indeterminacy locus of $X_{i-1} \drar X_i$ only finitely many times, where $i > 1$.
To this end, it is best to analyze the exceptional locus of the flipped contraction $X_i \rar W_{i-1}$, rather than the exceptional locus of the flipping contraction as in the previous paragraph.
Now, assume that $S_i$ intersects the exceptional locus of $X_i \rar W_{i-1}$.
In this scenario, we have two cases: either $D_i \subset S_i$ for some component of the exceptional locus for $X_i \rar W \subs i-1.$, or $S_i$ intersects the exceptional locus in finitely many points.}
{First,} assume $D_i \subset S_i$ is a component of the exceptional locus for $X_i \rar W \subs i-1.$.
As we have $B \subs S_i. \geq 0$, we have $a \subs D_i. (S_i,B \subs S_i. + M \subs S_i.) \leq 1$.
Since $D_i$ is in the flipped locus, it follows that $a \subs D_i. (S_1,B \subs S_1. + M \subs S_1.) < a \subs D_i. (S_i,B \subs S_i. + M \subs S_i.) \leq 1$.
By the reduction in the previous paragraph, we can assume that the generic point of the center of $D_i$ on $S_1$ is in the generalized {} {klt} locus of $(S_1, B \subs S_1. + M \subs S_1.)$.
{Furthermore, $D_i$ corresponds to a generalized non-canonical valuation of $(S_1,B_{S_1}+M_{S_1})$.}
Therefore, as $S_1$ is a surface, $D_i$ belongs to a finite set of divisors over $S_1$.
Furthermore, as the coefficients of $B \subs S_i.$ belong to a fixed DCC set, the coefficients along $D_i$ stabilize after finitely many steps.
Hence, as there are finitely many divisors involved and their coefficients all need to stabilize, $S_i$ can not contain any component of the exceptional locus of $X_i \rar W \subs i-1.$ for $i \gg 0$.
In particular, the birational maps $S_i \drar S \subs i+1.$ are morphisms for $i \gg 0$.
{Up to truncating the MMP, we may assume that $S_i \rar S_{i+1}$ is a morphism for all $i$.
To conclude, we need to address the second scenario, i.e., when $S_i$ intersects the exceptional locus of $X_i \rar W \subs i-1.$ in finitely many points.
In this case, $S_i$ is relatively ample for $X_i \rar W_{i-1}$, and thus $S_{i-1}$ is relatively anti-ample for $X \subs i-1. \rar W \subs i-1.$ and the exceptional locus of $X \subs i-1. \rar W \subs i-1.$ is contained in $S_{i-1}$.
In particular, the morphism $S \subs i-1. \rar S_i$ contracts some curves.
}
Since each $S_i$ is normal and the Picard rank drops if $S_{i-1} \rar S_i$ contracts a curve, these morphisms need to be isomorphisms for $i \gg 0$.
This shows that $X_{i-1} \drar X \subs i.$ is an isomorphism along $S_i$, concluding the proof.
\end{proof}

\subsection{$\pr 1.$-links}
In this section, we recall the notions of standard $\pr 1.$-link and $\pr 1.$-linkage, and we extend them to the context of generalized pairs.
Furthermore, we introduce a weaker version of these notions, which is sufficient for applications.

\begin{definition}
\label{def p1 link}
A generalized pair $(X,D_1+D_2+\Delta,\mathbf{M})/Z$ with a surjective proper morphism $g \colon X \rar T$ over $Z$ is called a {\em {weak} $\pr 1.$-link} if
\begin{enumerate}
	\item[(0)] $D_1$ and $D_2$ are distinct reduced prime divisors and $\lfloor \Delta \rfloor =0$;
    \item[(1)] $\K X. + D_1 + D_2 + \Delta \sim \subs \qq,g. 0$;
    \item[(2)] there exists a $\qq$-b-Cartier $\qq$-divisor $\mathbf N$ on $T$ such that $\bM. \sim_\qq g^* \mathbf N$ as $\mathbb Q$-b-divisors. That is, there is a birational model $g' \colon X' \rightarrow T'$ of $g \colon X \rightarrow T$ such that $(g')^*\mathbf{N}_{T'} \sim_\mathbb{Q}\bM X'.$, $\bM. = \overline{\bM X'.}$, and $\mathbf{N}=\overline{\mathbf{N}_{T'}}$;
    \item[(3)] $g \vert_{D_i} \colon \Supp(D_i) \rar T$ is an isomorphism for $i=1,2$;
    \item[(4)] $(X,D_1+D_2+\Delta,\mathbf{M})/Z$ is generalized plt; and
    \item[(5)] 
    the generic fiber of $g$ is isomorphic to $\pr 1.$, and $g$ is equidimensional of relative dimension 1.
\end{enumerate}
{Furthermore, a weak $\pr 1.$-link is called a {\em standard $\pr 1.$-link} if the following stronger version of (5) holds:}
\begin{itemize}
\item[($5'$)] {$(X_t)_{red}\cong\mathbb{P}^1$ for every $t\in T$.}
\end{itemize}
\end{definition}

\begin{example}
Let $(Z,\Delta_Z)$ be a klt pair.
Then, the pair
\[
(Z\times \pr 1.,\Delta_Z \times \pr 1. + Z \times \{ 0 \} + Z \times \{\infty \}) \rar Z
\]
is a standard $\pr 1.$-link.
\end{example}

\begin{definition}
\label{link.def}
Let $(X, B,\mathbf{M})/Z$ be a generalized dlt pair.
Assume that there exist a morphism $g \colon X \rar T$ over $Z$ such that $\K X. + B + \bM X. \sim \subs \qq,g. 0$.
Let $Z_1$, $Z_2$ be two generalized log canonical centers of $(X, B,\mathbf{M})/Z$.
We say that $Z_1$ and $Z_2$ are {\em directly {weakly} $\pr 1.$-linked} if there is a generalized log canonical center $W$ (or $W=X$ itself) satisfying the following conditions:
\begin{enumerate}
    \item $g (W) = g (Z_1) = g (Z_2)$;
    \item $Z_i \subset W$ for $i=1, 2$; and
    \item over a non-empty open subset of $g(W)$, the generalized pair $(W,B_W+N_W)/Z$ induced by repeated generalized divisorial adjunction {(see Remark \ref{dlt_adj})} onto $W$ is birational to a {weak} $\pr 1.$-link, with each $Z_i$ mapping to one the two horizontal sections of the weak $\pr 1.$-link structure.
\end{enumerate}
We extend the this relation to an equivalence relation as follows.
We say that $Z_1$ and $Z_2$ are {\em {weakly} $\pr 1.$-linked} if $Z_1=Z_2$ or there is a sequence of generalized log canonical centers $W_1, \ldots , W_n$ such that $W_1 = Z_1$, $W_n  =Z_2$ and $W_i$ is directly $\pr 1.$-linked to $W_{i+1}$ for $i=i, \ldots , n-1$.
{Finally, if every weak $\pr 1.$-link is also a standard $\pr 1.$-link, we drop ``weakly'' in the above definitions.}
\end{definition}

\section{Proof of the main statement} \label{sec3}
	
\subsection{Birational case}\label{sec3.1}
In this section, we prove some technical statements that are needed to understand how the non-klt locus behaves under birational modifications.

\begin{lemma} \label{lemma Nklt effective}
Let $(X,B)$ be a sub-pair {over $\Spec(R)$} with $\dim X \leq 3$, and let $\Delta \geq 0$ be an $\rr$-Cartier divisor.
Then, we have $\Nklt(X,B)=\Nklt(X,B+\epsilon \Delta)$ for $0 < \epsilon \ll 1$. 
\end{lemma}

\begin{proof}
Let $f \colon X' \rar X$ be a log resolution of $(X,\Supp(B+\Delta))$.
Set $\K X'. + B' = f^*(\K X. + B)$ and define $\Delta' \coloneqq f^*\Delta$.
By construction, $\Nklt(X,B + t \Delta) = f(\Nklt(X',B' + t \Delta '))$ for every $t \geq 0$.
Since $f$ is a log resolution, we have $\Nklt(X',B' + t \Delta')=\Supp((B'+t\Delta')\sups \geq 1.)$.
As $\Delta' \geq 0$, if $0 < \epsilon \ll 1$ we have $\Supp((B'+\epsilon \Delta')\sups \geq 1.) = \Supp((B')\sups \geq 1.)$.
Then, the claim follows.
\end{proof}

\begin{lemma} \label{lemma Nklt ample}
Let $(X,B)$ be a sub-pair {over $\Spec(R)$} with $\dim X \leq 3$, and let $H$ be an ample divisor.
Then, there exists $0 \leq \Delta \sim_\rr H$ such that $\Nklt(X,B)=\Nklt(X,B+\Delta)$.
\end{lemma}

\begin{proof}
We may take a normal projective compactification of $(X,B)$ to assume that $X$ is projective over $\Spec(R)$.

Let $f \colon X' \rar X$ be a log resolution of $(X,B)$ with an $f$-exceptional divisor $C \geq 0$ such that $-C$ is $f$-ample, which exists by Proposition \ref{prop:resolution}.
Note that the effectivity of $C$ follows from \cite{BMP+}*{Lemma 2.16}.
We write $K_{X'}+B'=f^*(K_X+B)$, where $f_*K_{X'}=K_X$.
By Lemma \ref{lemma Nklt effective}, $\Nklt(X',B')=\Nklt(X',B'+\mu C)$ for $0 < \mu \ll 1$.
Furthermore, we can ensure that, for the same choice of $\mu$, $f^*H - \mu C$ is ample.
	
By construction, $\Supp(B') \cup \Supp(C)$ is simple normal crossing.
Since $f^*H-\mu C$ is ample, we may find $0 \leq \Delta' \sim_\rr f^*H-\mu C$ such that $\lfloor \Delta' \rfloor = 0$, $\Supp(\Delta')$ shares no components with $\Supp(B'+C)$, and $B' + \mu C + \Delta'$ has simple normal crossing support, by \cite{BMP+}*{Remark 2.18}. 
Then, it follows that $\Nklt(X',B'+\mu C + \Delta')=\Nklt(X',B')$.
Define $\Delta \coloneqq f_*\Delta'$.
Since $C$ is $f$-exceptional, we have $\Delta \sim_\rr H$. Furthermore, we have 
\[
\K X'. + B' + \mu C + \Delta' = f^*(\K X. + B +\Delta) .
\]
Therefore, we have
\[
\Nklt(X,B+\Delta)=f(\Nklt(X',B'+ \mu C + \Delta'))= f(\Nklt(X',B')) = \Nklt(X,B),
\]
and the claim follows.
\end{proof}

\begin{lemma} \label{replacement_birkar}
{
Let $(X,B)$ be a $\mathbb{Q}$-factorial dlt pair over $\Spec(R)$.
Let $f \colon X \rar S$ be a birational contraction over $\Spec(R)$ such that $-(\K X. + B)$ is ample over $S$.
Then, $\Nklt(X,B)\cap f^{-1}(s)$ is connected as $k(s)$-scheme.}
\end{lemma}

\begin{proof}
Since $(X,B)$ is dlt, we have that $\Nklt(X,B)=\Supp \lfloor B \rfloor$.
Let $E \coloneqq \delta \lfloor B \rfloor$ for $0 < \delta \ll 1$ rational, so that, by Lemma \ref{lemma Nklt effective}, $\Nklt(X,B+E)=\Nlc(X,B+E)=\Supp \lfloor B \rfloor$.
Then, it suffices to prove the claim for $\Nlc(X,B+E)$.
Since $-(\K X. + B)$ is ample over $S$ and $0 < \delta \ll 1$, then so is $-(\K X. + B + E)$.
Thus, we may choose an $f$-ample divisor $M$ such that $\K X. + B + E + M \sim_{\mathbb{Q},f} 0$.

First, we claim that it is possible to run an $(K_X+B+M)$-MMP over $S$.
To see this, fix some ample divisor $A$ on $X$.
If $K_X+B+M$ is not nef over $S$, then we run a $(K_X+B+M)$-MMP with scaling of $A$.
For any $\epsilon>0$, we can write $K_X+B+M+\epsilon A\sim_{\mathbb{Q}} K_X+(1-\epsilon)B+\Delta_\epsilon$ where $\Delta_{\epsilon}\sim_{\mathbb{Q}} M+A+\epsilon B$ is chosen such that $(X,(1-\epsilon)B+\Delta_{\epsilon})$ is klt.
Therefore, we can perform a step of the $(K_X+B+M)$-MMP with scaling of $A$, and any finite sequence of steps of this MMP is also a $(K_X+(1-\epsilon)B+\Delta_\epsilon)$-MMP for $0<\epsilon\ll1$, and so we can continue to run it, while possibly shrinking $\epsilon$ repeatedly.

Since $\K X. + B + M \sim \subs \qq,f. -E$, this MMP is a $(-E)$-MMP.
Therefore, in every step of the MMP, the exceptional locus intersects $\Supp(E) \subset \Supp(\rddown{B})$.
Hence, this MMP terminates by Proposition \ref{prop:special_termination}, i.e., special termination for generalized pairs.
	
Notice that no connected component of $E$ can be completely contracted by this MMP, because every contraction is positive for $E$.
Also, we claim that distinct connected components of $\Supp(E)$ do not become connected by the MMP.
Indeed, suppose two such components $\Sigma$ and $\Xi$ do become connected by a step of the MMP $X_1\dashrightarrow X_2$. 
Let $S_1$ and $T_1$ be two prime components in $\Sigma$ and $\Xi$ on $X_1$, respectively, that are not contracted by $X_1\dashrightarrow X_2$, but their strict transforms $S_2$ and $T_2$ are connected by this step of the MMP.
Then, $(X_1,B_{X_1}-\epsilon(\rddown{B_{X_1}}-S_1-T_1),\overline{M})$ is generalized plt, and, for $\epsilon$ sufficiently small, $X_1 \drar X_2$ is also a step of the MMP for this generalized pair.
Then, we reach a contradiction, since the generalized plt property is preserved under the MMP, but $(X_2,B_{X_2}-\epsilon(\rddown{B_{X_2}}-S_2-T_2),\overline{M})$ is not generalized plt as $S_2 \cap T_2 \neq \emptyset$.
{Indeed, by \cite{birkar_complements}*{2.13.(2)}, the generalized dlt property is preserved by the MMP.
Then, if at any step of an MMP the generalized plt property were broken, it would mean that two disjoint prime components of the reduced part of the boundary would be connected, creating a non-empty intersection.
By the generalized dlt property, this new intersection would be a generalized log canonical center.
Thus, if the generalized plt property were not preserved, new generalized log canonical centers would appear, thus contradicting the negativity of the MMP.}

Since $f$ is a birational morphism, the MMP terminates with a minimal model $X_N$ on which $-E_{X_N}$ is nef over $S$, and no prime component of its support can dominate $S$.
Suppose $C$ is a curve in a fiber over $s$.
Then $C\cdot E_{Y_N} \leq 0$, that is, $C$ is either contained in $\Supp(E_{X_N})$ or completely disjoint from it.
Thus each fiber is entirely contained in $\Supp(E_{X_N})$ or completely disjoint from it. 
In particular, $\Supp(E_{X_N})$ is connected in the neighborhood of every fiber of $X_N \rar S$.
As the MMP $X \drar X_N$ preserves the connected components of $E$, this concludes the proof.
\end{proof}

We recall the birational version of the three-dimensional Shokurov--Koll\'{a}r connectedness principle, which is due to Nakamura and Tanaka {in the case of a perfect field of positive characteristic} \cite{NT20}*{Theorem 2.15}.

\begin{lemma}[cf. \cite{NT20}*{Theorem 2.15}]\label{lem:ample_to_big_nef}
Let $f \colon X \rar V$ be a projective birational morphism of normal quasi-projective varieties of dimension at most 3 over {$\Spec(R)$}.
Let $(X,\Delta)$ be a sub-pair over {$\Spec(R)$} such that $-(\K X. + \Delta)$ is $f$-nef and $f_* \Delta$ is effective.
Then, the induced morphism $\Nklt(X,\Delta) \rar V$ has connected fibers.
\end{lemma}

\begin{proof}
	This is proved in the case of a variety over a perfect field in \cite{NT20}*{Theorem 2.15}.
	The proof given there goes through verbatim in our situation except for the use of \cite{birkar_p}*{Theorem 1.8} in the proof of \cite{NT20}*{Lemma 2.13}.
	Indeed, \cite{birkar_p}*{Theorem 1.8} uses \cite{birkar_p}*{Lemma 9.2}, for which the mixed characteristic replacement \cite{BMP+}*{Lemma 2.34} is weaker.
	That is because log resolutions in mixed characteristic cannot be guaranteed to be isomorphisms over the regular locus.
	This issue is overcome with Lemma \ref{replacement_birkar}, which replaces  \cite{birkar_p}*{Theorem 1.8} in the proof of \cite{NT20}*{Lemma 2.13}.
	While the statement of \cite{NT20}*{Theorem 2.15} only refers to threefolds, the connectedness principle clearly holds for surfaces: it suffices to take the product with $\pr 1.$ to reduce to the threefold case.
\end{proof}

\begin{proposition} \label{MMp gpair connectedness}
Let $(X,B,\bM.)/Z$ be a generalized pair with $\dim X \leq 3$ over $\Spec(R)$.
Let $\pi \colon X \rar Y$ be a birational contraction over $Z$, with $Y$ normal.
Assume that $-(\K X. + B + M)$ is $\pi$-nef.
Then, $\Nklt(X,B,\bM.)$ is connected in a neighborhood of any fiber of $\pi$.
\end{proposition}

\begin{proof}
Since $\pi$ is birational, $-(\K X. + B +M)$ is $\pi$-big. 
Thus, we have $-(\K X. + B + M) \sim \subs \rr,\pi. A$, where $A$ is $\pi$-nef and $\pi$-big.
Let $f \colon X' \rar X$ be a higher model of $X$ to which $\bM.$ descends.
Let $(X',B'+M')$ denote the {crepant pull-back} of $(X,B+M)$ on $X'$, and write $A' \coloneqq f^*A$.
Notice that $\Nklt(X',B',\bM.)=\Nklt(X',B')$, as $\bM.$ descends onto $X'$.
	
Since $A'+M'$ is nef and big over $Y$, there exists $E' > 0$ such that for any $\epsilon > 0$ {sufficiently small} we have $A'+M' \sim_{\rr,Y} H'_\epsilon + \epsilon E'$, where $H'_\epsilon$ is ample.
By Lemma \ref{lemma Nklt effective}, $\Nklt(X',B'+\epsilon E') = \Nklt(X',B')$ for $0 < \epsilon \ll 1$.
Fix such $\epsilon$.
Then, by Lemma \ref{lemma Nklt ample}, there exists $0 \leq \Delta'_\epsilon \sim_\rr H'_\epsilon$ such that $\Nklt(X',B'+\epsilon E' + \Delta'_\epsilon) = \Nklt(X',B'+\epsilon E')$.
Define $\Gamma' \coloneqq \Delta'_\epsilon + \epsilon E'$ and $\Gamma \coloneqq f_* \Gamma'$.
	
By construction, we have $\K X'. + B' + \Gamma' = f^*(\K X. + B + \Gamma)$
and $\Nklt(X,B+\Gamma) = f(\Nklt(X',B'+\Gamma'))$.
Since
\[
\Nklt(X',B' + \Gamma') = \Nklt(X',B') = \Nklt(X',B',\bM.),
\]
we have \[
\Nklt(X,B,\bM.)=f(\Nklt(X',B',\bM.))=f(\Nklt(X',B'))=f(\Nklt(X',B'+\Gamma'))=\Nklt(X,B+\Gamma).
\]
Since $\K X. + B + \Gamma \sim \subs \rr,\pi. 0$, we conclude by Lemma \ref{lem:ample_to_big_nef}.
\end{proof}

\subsection{Connectedness of the non-lc locus} \label{sec3.2}
In this section, we address the connectedness of the non-log canonical locus.

\begin{proposition}\label{prop:nlc_vertical}
Let $f \colon X \rar S$ be a contraction between normal varieties defined {over $\Spec(R)$}.
Assume that $\dim X \leq 3$ and that $(X,B)$ is a pair with rational coefficients and $-(\K X. + B)$ nef over $S$.
Assume that $\Nlc(X,B)$ is vertical over $S$.
Then, $\Nlc(X,B)\cap f^{-1}(s)$ is connected as $k(s)$-scheme.
\end{proposition}

\begin{proof}
We follow the argument of \cite{birkar_p}*{Theorem 1.8}, except in our case it is easier.
	
Let $\phi \colon Y\to X$ be a $\bQ$-factorial dlt model of $(X,B)$.
That is to say, let $\psi \colon W\to X$ be a log resolution of $(X,B)$, and run a $(K_W+{\widetilde{B}}^{\leq 1}_W+F)$-MMP$/X$, where $\widetilde{B}_W$ is the strict transform of $B$, $\widetilde{B}_W^{\leq 1}$ indicates truncation of the coefficients at $1$, and $F$ is the reduced $\psi$-exceptional divisor.
We reach a $\mathbb{Q}$-factorial model $Y$ on which $K_Y+\widetilde{B}_Y^{\leq 1}{+F_Y}$ is nef over $X$ and $(Y, \widetilde{B}_Y^{\leq 1}+F_Y)$ is dlt, where again $\widetilde{B}_Y$ is the birational transform of $B$, and $F_Y$ is the reduced $\phi$-exceptional divisor.
Now, we define $B_Y \coloneqq \widetilde{B}_Y^{\leq 1}+F_Y$, so that $(Y,B_Y)$ is dlt, and we set $E_Y \coloneqq \phi^*(K_X+B)-(K_Y+B_Y^{\leq 1}+F_Y)$.
By construction, $E_Y$ is supported on the strict transform of $\lfloor B \rfloor$ and on the $\phi$-exceptional divisors.
Thus, we have that $\Supp(E_Y) \subset \lfloor B_Y \rfloor$.
Furthermore, since $-E_Y$ is nef over $X$ and $\phi_*E_Y=B-B^{\leq 1}\geq 0$, $E_Y\geq 0$ by the negativity lemma.
Therefore, since $E_Y \geq 0$, $\Supp(E_Y) \subset \lfloor B_Y \rfloor$, and $(Y,B_Y)$ is dlt, it follows that $\Supp(E_Y)=\Nlc(Y,B_Y+E_Y)$.
For brevity, we write $g = f \circ \phi$.
Let $M$ be a $g$-nef divisor such that $K_Y+B_Y+E_Y+M\sim_{\bQ,g}0$.
Since $\phi(\mathrm{Nlc}(Y,B_Y+E_Y)\cap g^{-1}(s))=\mathrm{Nlc}(X,B)\cap f^{-1}(s)$, it is enough to prove the statement for $(Y,B_Y+E_Y)$.

First, we claim that it is possible to run an $(K_Y+B_Y+M)$-MMP over $S$.
To see this, fix some ample divisor $A$ on $Y$.
If $K_Y+B_Y+M$ is not nef over $S$, then we run a $(K_Y+B_Y+M)$-MMP with scaling of $A$.
For any $\epsilon>0$, we can write $K_Y+B_Y+M+\epsilon A\sim_{\mathbb{Q}} K_Y+(1-\epsilon)B_Y+\Delta_\epsilon$ where $\Delta_{\epsilon}\sim_{\mathbb{Q}} M+A+\epsilon B_Y$ is chosen such that $(Y,(1-\epsilon)B_Y+\Delta_{\epsilon})$ is klt.
Therefore, we can perform a step of the $(K_Y+B_Y+M)$-MMP with scaling of $A$, and any finite sequence of steps of this MMP is also a $(K_Y+(1-\epsilon)B_Y+\Delta_\epsilon)$-MMP for $0<\epsilon\ll1$, and so we can continue to run it, while possibly shrinking $\epsilon$ repeatedly.

Since $\K Y. + B_Y + M \sim \subs \qq,g. -E_Y$, this MMP is a $(-E_Y)$-MMP.
Therefore, in every step of the MMP, the exceptional locus intersects $\Supp(E_Y) \subset \Supp(\rddown{B_Y})$.
Hence, this MMP terminates by Proposition \ref{prop:special_termination}, i.e., special termination for generalized pairs.
	
Notice that no connected component of $E_Y$ can be completely contracted by this MMP, because every contraction is positive for $E_Y$.
Also, we claim that distinct connected components of $\Supp(E_Y)$ do not become connected by the MMP.
Indeed, suppose two such components $\Sigma$ and $\Xi$ do become connected by a step of the MMP $Y_1\dashrightarrow Y_2$. 
Let $S_1$ and $T_1$ be two prime components in $\Sigma$ and $\Xi$ on $Y_1$, respectively, that are not contracted by $Y_1\dashrightarrow Y_2$, but their strict transforms $S_2$ and $T_2$ are connected by this step of the MMP.
Then, $(Y_1,B_{Y_1}-\epsilon(\rddown{B_{Y_1}}-S_1-T_1),\overline{M})$ is generalized plt, and, for $\epsilon$ sufficiently small, $Y_1 \drar Y_2$ is also a step of the MMP for this generalized pair.
Then, we reach a contradiction, since the generalized plt property is preserved under the MMP, but $(Y_2,B_{Y_2}-\epsilon(\rddown{B_{Y_2}}-S_2-T_2),\overline{M})$ is not generalized plt as $S_2 \cap T_2 \neq \emptyset$.
{Indeed, by \cite{birkar_complements}*{2.13.(2)}, the generalized dlt property is preserved by the MMP.
Then, if at any step of an MMP the generalized plt property were broken, it would mean that two disjoint prime components of the reduced part of the boundary would be connected, creating a non-empty intersection.
By the generalized dlt property, this new intersection would be a generalized log canonical center.
Thus, if the generalized plt property were not preserved, new generalized log canonical centers would appear, thus contradicting the negativity of the MMP.}

Since $E_Y$ does not dominate $S$ by assumption, the MMP terminates with a minimal model $Y_N$ on which $-E_{Y_N}$ is nef over $S$.
In particular, no prime component of its can dominate $S$.
Suppose $C$ is a curve in a fiber over $s$.
Then $C\cdot E_{Y_N} \leq 0$, that is, $C$ is either contained in $\Supp(E_{Y_N})$ or completely disjoint from it.
Thus each fiber is entirely contained in $\Supp(E_{Y_N})$ or completely disjoint from it. 
In particular, $\Supp(E_{Y_N})$ is connected in the neighborhood of every fiber of $Y_N \rar S$.
As the MMP $Y \drar Y_N$ preserves the connected components of $E_Y$, this concludes the proof.
\end{proof}

\subsection{Main proof} \label{sec3.3}
In this section, we prove the main results of this work.

\begin{proposition}\label{prop:nklt_dominates}
Let $f \colon X \rar S$ be a projective morphism between quasi-projective varieties defined {over $\Spec(R)$}.
Assume that $\dim X \leq 3$ and that $(X,B)$ is a pair with rational coefficients and $-(\K X. + B)$ nef over $S$.
Assume that $\Nklt(X,B)$ is vertical over $S$, and that $f \sups -1. (s)$ is connected as $k(s)$-scheme.
Then, $f \sups -1. (s) \cap \Nklt(X,B)$ is connected as $k(s)$-scheme. 
\end{proposition}

\begin{proof}
First, we reduce to the case when $f$ is a contraction.
Arguing by contradiction, assume that $f \sups -1. (s) \cap \Nklt(X,B)$ is disconnected as $k(s)$-scheme.
Let $\phi \colon X \rar Y$ denote the Stein factorization of $f$, and let $\psi \colon Y \rar S$ denote the induced finite morphism.
Since the fiber of $f$ over $s$ is connected, and $\phi$ has geometrically connected fibers, there exists a unique point $y \in Y$ mapping to $s \in S$.
Furthermore, as $f \sups -1. (s) \cap \Nklt(X,B)$ is disconnected, then so is $\phi \sups -1. (y) \cap \Nklt(X,B)$.
Thus, the hypotheses of the statement apply to the morphism $\phi \colon X \rar Y$.
Furthermore, $\Nklt(X,B)$ dominates $S$ if and only if it dominates $Y$.
Therefore, we may assume that the morphism $f$ satisfies the additional property $f_* \O X. = \O S.$.

Let $\pi \colon Y\to X$ be a crepant $\bQ$-factorial dlt model, $B_Y$ and $E_Y$ be effective divisors such that $(Y,B_Y)$ is dlt, $\Supp(E_Y) \subset \lfloor B_Y \rfloor$,
\[
\pi^*(K_X+B)=K_Y+B_Y+E_Y,
\]
and $\Supp(E_Y)=\Nlc(Y,B_Y+E_Y)$.
Since
\[
\pi(\Nklt(Y,B_Y+E_Y)\cap (f\circ\pi)^{-1}(s)) =\Nklt(X,B)\cap f^{-1}(s),
\]
it is enough to prove the statement for $(Y,B_Y+E_Y)$.
For brevity, we write $g= f \circ \pi$.	

In order to prove the statement, we are free to replace $(s \in S)$ by an \'etale cover $(s' \in S') \rar (s \in S)$ such that $k(s)=k(s')$. Thus, arguing as in \cite{kollar_singularities}*{4.38}, we can assume that different connected components of $\Nklt(Y,B_Y+E_Y) \cap g \sups -1. (s)$ correspond to different connected components of $\Nklt(Y,B_Y+E_Y)$.

Notice that, by Lemma \ref{lem:ample_to_big_nef}, we may assume that $\dim S < \dim X$.
Furthermore, the assumption that $\Nklt(X,B)$ is vertical over $S$ forces $\dim S > 0$.
Let $D$ be an effective Cartier divisor on $S$ containing the image under $g$ of all components of $\Nklt(Y,B_Y+E_Y)$ which intersect $g^{-1}(s)$.
Then, let $F=g^*D$ and choose $0< \epsilon \ll 1$ so that $\Nklt(Y,B_Y+E_Y)=\Nklt(Y,B_Y+E_Y+\epsilon F)$.
Then, as $(Y,B_Y)$ is dlt, $\Supp(E_Y) \subset \lfloor B_Y \rfloor$, and $\lfloor B_Y \rfloor \subset \Supp (F)$ in a neighborhood of $g \sups -1. (s)$, it follows that $\Nklt(Y,B_Y+E_Y+\epsilon F)=\Nlc(Y,B_Y+E_Y+\epsilon F)$ in a neighborhood of $g \sups -1. (s)$.
Hence, we are done by Proposition \ref{prop:nlc_vertical}.
\end{proof}

Now, we are ready to address the proof of Theorem \ref{thm:main-thm}.

\begin{proof}[Proof of Theorem \ref{thm:main-thm}]
We will proceed in several steps.

{\bf Step 0:} In this step, we make a few reductions and introduce some notation.

We may assume that $\dim(S)<\dim(X)$ by Lemma \ref{lem:ample_to_big_nef}.
As argued in the proof of Proposition \ref{prop:nklt_dominates}, we may assume that $f \colon X \rar S$ is a contraction.

Let $\pi \colon Y \to X$ be a $\qq$-factorial dlt model of $(X,B)$, and let $g \colon Y\to S$ denote the induced contraction.
Let $(Y,B_Y)$ denote the trace of $(X,B)$ on $Y$, that is, $K_Y+B_Y=\pi^*(K_X+B)$, and define $M \coloneqq -(\K Y. + B_Y)$.
We define $D \coloneqq B_Y \sups \geq 1.$ and $\Delta \coloneqq B_Y - D$.
Then, $D$ and $\Delta$ share no prime components and $(Y,\Supp(D) + \Delta)$ is dlt.
Up to a base change of $S$ with an \'etale neighborhood of $s \in S$, we may assume that each connected component of $g^{-1}(s)\cap \Supp (D)$ corresponds to a connected component of $D$ \cite{kollar_singularities}*{4.38}.
Then, by Lemma \ref{lem:ample_to_big_nef}, we can reduce to the study of $g \colon (Y,B_Y) \rar S$.
We will show that $(Y,B_Y)$ is plt in a neighborhood of $g^{-1}(s)$.

{\bf Step 1:} In this step, we show that $g^{-1}(s)\cap\Nklt(Y,B_Y)$ has two connected components.

Since $M$ is nef over $S$, $(Y,B_Y,\overline M)/S$ is a generalized pair.
{Furthermore, as the $\mathbb Q$-b-Cartier $\mathbb Q$-b-divisor $\overline{M}$ descends onto $Y$, we have $\Nklt(Y,B_Y,\overline{M})=\Nklt(Y,B_Y)=\Supp(D)$.}
By Proposition \ref{prop:nklt_dominates}, $\Supp(D)$ dominates $S$.
Then, $\K Y. + \Delta + M$ is not pseudo-effective over $S$.
Therefore, we can run a $(\K Y. + \Delta + M)$-MMP with ample scaling relative to $S$.
By Proposition \ref{prop:non-pseff-termination-2}, this MMP terminates with a Mori fiber space $h \colon Y' \rar Z$ over $S$.
Let $B_{Y'}$ and $M'$ denote the push-forwards of $B$ and $M$ along the rational contraction $Y \drar Y'$.
Since $K_Y+B_Y+M\sim_{\mathbb{Q},g}0$, it follows that this procedure induces a generalized pair $(Y',B_{Y'},\overline{M})$ that is crepant to $(Y,B_Y,\overline{M})$; see \cite{BZ16}*{Remark 4.2.(7)}.
Notice that we have $\overline{M}_{Y'}=M'$.
Since $\K {Y}. + B_{Y} + M \sim \subs \qq,g. 0$, by Proposition \ref{MMp gpair connectedness}, the connected components of $\Nklt(Y,B_Y,\overline M)$ are in bijection with the connected components of $\Nklt(Y',B_{Y'},\overline M)$.
In particular, the MMP cannot connect different connected components of $\Nklt(Y,B_Y, \overline M) = \Nklt(Y,B_Y)$.

By construction, the MMP is a $(-D)$-MMP.
Therefore, it cannot contract every divisor in a given connected component of $\Supp(D)$.
In particular, every connected component of $\Nklt(Y',B_{Y'},\overline M)$ contains a divisor, and one of these dominates $Z$ and is $h$-ample.
Call this divisor $D'_1$, and notice that, since $Y' \rar Z$ is a Mori fiber space and hence $\rho(Y'/Z)=1$, we may choose it to be reduced and irreducible.
Let $D'_2 \subset \Nklt(Y',B_{Y'},\overline{M})$ be a prime divisor contained in a different connected component.
Recall that this implies $D_1'$ and $D_2'$ are irreducible components of $D'$, {where $D'$ denotes the strict transform of $D$}.
Then, as $D_1'$ is $h$-ample and $D_1' \cap D_2' = \emptyset$, $D_2'$ contains no fiber.
Thus, by dimensional reasons, $D_2'$ dominates $Z$.
Since $h$ is a Mori fiber space, it follows that $D_2'$ is $h$-ample.

Since $D_1' \cap D_2' = \emptyset$, every fiber of $h$ is a curve.
As $\chara (k) > 2$, the geometric generic fiber is $\pr 1.$ by \cite{cascini_tanaka_xu}*{Lemma 6.5}.
In particular, it follows that there are exactly two components $D_1'$ and $D_2'$ of $\Nklt(Y',B_{Y'}, \overline M)$, and that $\mu \subs D_1'. (B_{Y'}) = \mu \subs D_2'. (B_{Y'})=1$ {(where we recall from \S\ref{sub:notation} that $\mu_{D}(B)$ denotes the coefficient of the prime divisor $D$ in $B$)}, $(B_{Y'})^h=D_1'+D_2'$ and that $M'$ is trivial along the generic fiber {of $h$}.
Since we are assuming that $\Nklt(Y',B \subs Y'.,\overline M) \cap Y' \subs s.$ is disconnected, the $h$-vertical components of $D'$ have to be disjoint from $Y' \subs s.$.
Hence, up to shrinking around $s \in S$, we have $B \subs Y'. \sups \geq 1.= B \subs Y'. \sups =1. = D_1' + D_2'$.
In particular, as every connected component of $\Nklt(Y',B_{Y'},\overline M)$ contains a divisor, this shows that $\Nklt(Y,B_Y)$ has exactly two connected components over $s \in S$.

{\bf Step 2:} In this step we show that $(Y,B_Y)$ is log canonical in a neighborhood of $g \sups -1.(s)$ and hence dlt in a neighborhood of $g \sups -1.(s)$.

Assume by contradiction that $(Y,B_Y)$ is not log canonical {in a neighborhood of $g \sups -1.(s)$}.
By Step 1, we know that $\Nklt(Y,B_Y)$ has two connected components over $s \in S$.
Call these components $\Omega_1$ and $\Omega_2$.
From Step 1, we may assume that $D_1 \subset \Omega_1$ and $D_2 \subset \Omega_2$.
Here, $D_1$ and $D_2$ denote the strict transforms on $Y$ of the divisors $D_1'$ and $D_2'$ constructed in Step 1.
Without loss of generality, we may assume that $\Omega_1$ contains some non-log canonical center; that is, we may assume that some prime divisor in $\Omega_1$ has coefficient strictly greater than 1 in $B_Y$.

Now, let $\Sigma$ be the portion of $B_Y \sups \geq 1.$ supported on $\Omega_2$.
Then, let $\Xi$ be the portion of $B_Y \sups > 1.$ supported on $\Omega_1$.
Finally, let $\Gamma$ be the portion of $B_Y \sups =1.$ supported on $\Omega_1$.
Then, we have
\[
\K Y. + \Delta + \Gamma + M \sim \subs \qq,g. -\Sigma - \Xi.
\]
Since $D_2 \subset \Supp(\Sigma)$ dominates $S$, we conclude that $\K Y. + \Delta + \Gamma + M$ is not pseudo-effective over $S$.
Since $(Y,\Delta+\Gamma)$ is dlt and $M$ is nef over $S$, by Proposition \ref{prop:non-pseff-termination-2} we can run a $(\K Y. + \Delta + \Gamma + M)$-MMP with scaling over $S$, which terminates with a Mori fiber space $Y'' \rar W$.

By Proposition \ref{MMp gpair connectedness}, $\Omega_1$ and $\Omega_2$ cannot be connected by this MMP.
Since $\Supp(\Sigma) \subset \Omega_2$ and $\Supp(\Xi) \subset \Omega_1$ are disconnected and this is a $-(\Sigma + \Xi)$-MMP, the divisors $\Sigma$ and $\Xi$ are not completely contracted {and the respective birational transforms $\Sigma''$ and $\Xi''$ are non-zero}.
Now, arguing as in Step 1 for the construction of the divisors $D_1'$ and $D_2'$, it follows that $Y'' \rar W$ is a Mori fiber space with  geometric generic fiber $\pr 1.$, and there are two distinct prime divisors $P_1''$ and $P_2''$ that are relatively ample and that are in the supports of $\Xi''$ and $\Sigma''$, respectively.
Since $\K Y''. + B \subs Y''. + M'' \sim \subs \qq,W. 0$, $\deg (K_{\pr 1.})=-2$, $M''$ is pseudo-effective over $S$, $\mu \subs P_1''.(B \subs Y''.) >1$ and $\mu \subs P_2''.(B \subs Y''.) \geq 1$, we get the sought contradiction by restricting to the geometric generic fiber of $Y'' \rightarrow W$.

{\bf Step 3:} In this step we show that $(Y,B_Y)$ is plt in a neighborhood of $g \sups -1.(s)$.

By Step 2, we know that $(Y,B_Y)$ is dlt in a neighborhood of $g^{-1}(s)$.
In particular, $B_Y \sups \geq 1.= B_Y \sups =1.$.
Recall that $\Nklt(Y,B_Y)$ has two connected components, denoted by $\Omega_1$ and $\Omega_2$.
Furthermore, each $\Omega_i$ contains a prime divisor $D_i$ that dominates $S$.

Now, assume by contradiction that $(Y,B_Y)$ is not plt.
Then, some $\Omega_i$ contains at least two prime divisors.
Without loss of generality, we may assume it is $\Omega_1$.
Let $Q_1$ be this additional prime divisor contained in $\Omega_1$ and distinct from $D_1$. 
As in Step 2, let $\Sigma$ denote the portion of $B_Y \sups =1. = B_Y \sups \geq 1.$ supported on $\Omega_2$.
Let $\Phi$ denote the portion of $B_Y \sups =1.$ supported on $\Omega_1$.
Notice that $D_1+Q_1 \leq \Phi$.
Then, we have
\[
\K Y. + \Delta + \Phi + M \sim \subs \qq,g. -\Sigma.
\]
By Proposition \ref{prop:non-pseff-termination-2}, we may run a $(\K Y. + \Delta + \Phi + M)$-MMP over $S$ with ample scaling, which terminates with a Mori fiber space $u \colon \tilde{Y} \rar U$ over $S$.

As argued in Step 1, $\Omega_1$ and $\Omega_2$ cannot be connected by this MMP.
As the MMP $Y \drar \tilde{Y}$ is a $(-\Sigma)$-MMP, $\Sigma$ cannot be fully contracted: there is some prime component $Q_2$ of $\Sigma$ whose strict transform $\tilde Q _2$ on $\tilde Y$ is $u$-ample.
Now, let $Y_i \rar Y \subs i+1.$ be a divisorial contraction for the MMP $Y \drar \tilde Y$, and let $E_i$ be the exceptional divisor for the contraction.
Then, $E_i$ is covered by curves that are positive against $\Sigma_i$.
Since $\Supp(\Phi_i) \cap \Supp(\Sigma_i) = \emptyset$, it follows that $E_i$ cannot be a prime component of $\Phi_i$.
In particular, no component of $\Omega_1$ is contracted by $Y \drar \tilde Y$.
In particular, $\tilde D _1 \neq 0$ and $\tilde Q _1 \neq 0$, where $\tilde D _1$ and $\tilde Q_1$ denote the strict transforms on $\tilde Y$ of $D_1$ and $Q_1$, respectively.

Since $\tilde Q _2$ is $u$-ample and it is disjoint from $\tilde D _1 + \tilde Q _1$, as in Step 1, it follows that $u$ is generically a $\pr 1.$-fibration and that both $\tilde D _1$ and $\tilde Q _1$ are horizontal over $U$.
Now, as $\K \tilde Y. +  B \subs \tilde Y. + \tilde M \sim \subs \qq,u.0$, $\deg(K_{\pr 1.})=-2$, $\tilde M$ is pseudo-effective over $S$, $\mu \subs \tilde D _1.(B \subs \tilde Y.)=1$, $\mu \subs \tilde Q _1.(B \subs \tilde Y.)=1$, and $\mu \subs \tilde Q _2.(B \subs \tilde Y.)=1$, we get the sought contradiction.

{\bf Step 4:} In this step, we show that $(Y',B_{Y'},\overline{M})/Z$ satisfies {conditions (0)-(5) of} Definition \ref{def p1 link}, where $Y'$ is the model obtained in Step 1.

By the end of Step 1, we have that $B_{Y'}^{\geq 1}=D_1'+D_2'$, where $D_1'$ and $D_2'$ are distinct reduced prime divisors.
This shows that condition (0) in Definition \ref{def p1 link} holds.
By Step 3, $(Y,B_Y)$ is plt.
Since $Y \drar Y'$ is trivial for $\K Y. + B_Y + M$ and it does not contract the two log canonical centers, by Lemma \ref{lem:plt_crepant}, it follows that $(Y',B_{Y'},\overline{M})/Z$ is generalized plt.
Thus, condition (4) in Definition \ref{def p1 link} is satisfied.

By construction, $h \colon Y' \rar Z$ is a Mori fiber space with the property that $\K Y'. + B \subs Y'. \sups =1. \sim \subs \qq,Z. 0$.
As the generalized pair $(Y',B_{Y'},\overline{M})/Z$ is generalized plt and $Y'$ is $\qq$-factorial, the pair $(Y',B \subs Y'.)$ is plt.
Furthermore, we have $\K Y'. + B \subs Y'. + M' \sim \subs \qq,Z. 0$, where we recall that $M'=\overline{M}_{Y'}$ by definition.
Therefore we have $B_{Y'}-B_{Y'}^{=1}\sim_{\qq,Z} -M'$, and hence, as the left-hand side is effective and $M'$ is pseudo-effective over $Z$, by \cite{GNT}*{Theorem 2.9}, we have $B_{Y'}-B_{Y'}^{=1}\sim_{\qq,Z}-M'\sim_{\qq,Z}0$.
Thus, condition (1) in Definition \ref{def p1 link} is satisfied.

Since $(Y',B \subs Y'.)$ is plt and $Y'$ is $\mathbb{Q}$-factorial, each irreducible component of $B_{Y'}^{=1}$ is normal by \cite{BMP+}*{Corollary 7.17}.
Since neither component contains any fiber, the induced morphism to $Z$ is finite and birational.
As $Z$ is normal, this is an isomorphism.
So, condition (3) in Definition \ref{def p1 link} is satisfied.

Since $h$ is a Mori fiber space and it has two disjoint sections, which are necessarily relatively ample, $h$ is equidimensional of relative dimension 1.
Since the generic point of $Z$ has characteristic 0 or strictly greater than 5, and the generic fiber of $g$ has two distinct rational points, it follows that the generic fiber is isomorphic to $\pr 1.$.
Thus, (5) follows.

Now, let $Y^r \rar Y'$ be a common resolution of $Y$ and $Y'$.
In particular, $\overline{M}$ descends onto $Y^r$.
As $Y' \rar Z$ has relative dimension 1, $Y^r \rar Y'$ is an isomorphism over a non-empty open subset of $Z$.
Let $Y^f \rar Z^f$ be a flattening of $Y^r \rar Z$, see {\cite{raynaud_gruson}*{Theorem 5.2.2}}, and let $\widetilde{Z}$ be a resolution of $Z^f$.
Now let $\widetilde{Y}$ be the normalization of the main component of $Y^f\times_{Z^f}\widetilde{Z}$.
This ensures that $\widetilde{Y}\to\widetilde{Z}$ still has equidimensional fibers, so by the proof of \cite{Wit17}*{Lemma 2.18} in our more general situation, it follows that $\overline{M}_{\widetilde{Y}} \sim \subs \qq,\widetilde{Z}. 0$ 
Thus, condition (2) in Definition \ref{def p1 link} is satisfied.

{\bf Step 5:} In this step, we conclude the proof by showing that, in the equicharacteristic case with perfect residue field, condition ($5'$) of Definition \ref{def p1 link} holds.
{Throughout this step, we assume that $f \colon X \rar S$ is a projective morphism between varieties defined over a perfect field of characteristic $p>5$.}

We continue with the notation of Step 4.
To show condition {($5'$)} in Definition \ref{def p1 link} we argue as in the proof of \cite{kollar_singularities}*{Proposition 4.37}.
By \cite{Ber20}*{Theorem 22}, 
we have that $R^i h_* \O Y'.=0$ for $i \geq 1$.
Thus, $H^1(Y'_t,\O Y'_t.)=0$ for every $t \in Z$ by Lemma \ref{lemma_tree_curves}.
Thus, $(Y'_t)_\mathrm{red}$ is a tree of smooth rational curves for every $t \in Z$.
Since all the connected components of $Y'_t$ have to meet at the two intersection points of $Y'_t \cap B \subs Y'. \sups =1.$, it follows that the tree consists of just one curve.
Thus, condition {($5'$)} in Definition \ref{def p1 link} holds.
\end{proof}

\begin{lemma}\label{lemma_tree_curves}
Let $f \colon X \rar Y$ be a projective morphism of noetherian schemes with $Y$ integral.
Assume that $\dim X_y=1$ for all $y \in Y$ and that $R^1f_* \mathcal{O}_X=0$.
Then, we have $H^1(X_y,\mathcal{O}_{X_y})=0$ for all $y \in Y$.
\end{lemma}

\begin{proof}
Since the question is local on the base, we may assume that $Y$ is affine.
Now, let $Z \subset Y$ be a closed subset, and let $X_Z$ denote $X \times_Y Z$.
The short exact sequence
\[
0 \rar \mathcal{I}_{X_Z} \rar \O X. \rar \O X_{Z}. \rar 0
\]
induces a long exact sequence of higher direct images
\[
\ldots R^1f_* \O X. \rar R^1 f_* \O X_{Z}. \rar R^2 f_*\mathcal{I}_{X_{Z}} \rar \ldots.
\]
By \cite{Har77}*{Corollary III.11.2} and the assumption on the dimension of the fibers, we have $R^2 f_*\mathcal{I}_{X_{Z}}=0$.
Then, by $R^1f_*\O X.=0$, it follows that $R^1 f_* \mathcal{O}_{X_Z} =0$.
Notice that this settles the statement in the case $y$ is a closed point. 

The above shows that the assumptions of the statement are preserved if we replace $Y$ and $X$ by $Z$ and $X_Z$, where $Z \subset Y$ is an integral closed subscheme.
Now, let $y \in Y$ be a point.
By the previous paragraph, we may assume that $y$ is not a closed point.
Then, let $\overline{\{ y \}}$ denote the closure of $\{y\}$ in $Y$.
Since $\overline{\{ y\}}$ is an integral subscheme of $Y$, by the above argument, we may replace $Y$ with $\overline{\{ y\}}$.
Thus, up to relabelling, we may assume that $y$ is the generic point of $Y$.
Since $Y$ is integral, we may apply generic flatness.
So, up to shrinking $Y$, we may assume that $X \rar Y$ is flat.
Then, as the claim is settled over closed points, by upper semi-continuity \cite{Har77}*{Theorem III.12.8}, it follows that $H^1(X_y,\O X_y.)=0$.
\end{proof}

We conclude with the proofs of the corollaries to the main statement.

\begin{proof}[Proof of Corollary \ref{P1_link}]
The proof of \cite{filipazzi_svaldi}*{Theorem 1.4} goes through verbatim.
{To this end, we observe that the notion of weak $\pr 1.$-link suffices for the argument to go through.}
\end{proof}

\begin{proof}[Proof of Corollary \ref{kollar_441}]
{In view of Corollary \ref{P1_link},} this follows exactly the same argument as \cite{kollar_singularities}*{Corollary 4.41}.
\end{proof}


\begin{bibdiv}
\begin{biblist}

\bib{ahk}{article}{
   author={Alexeev, Valery},
   author={Hacon, Christopher D.},
   author={Kawamata, Yujiro},
   title={Termination of (many) 4-dimensional log flips},
   journal={Invent. Math.},
   volume={168},
   date={2007},
   number={2},
   pages={433--448},
   issn={0020-9910},
   review={\MR{2289869}},
   doi={10.1007/s00222-007-0038-1},
}

\bib{Ber20}{article}{
   author={Bernasconi, Fabio},
   author={Koll\'{a}r, J\'{a}nos},
   title={Vanishing theorems for three-folds in characteristic $p>5$},
   journal={Int. Math. Res. Not. IMRN},
   date={2023},
   number={4},
   pages={2846--2866},
   issn={1073-7928},
   review={\MR{4565629}},
   doi={10.1093/imrn/rnab316},
}

\bib{BMP+}{article}{
 author = {Bhatt, Bhargav},
 author = {Ma, Linquan},
 author = {Patakfalvi, Zsolt},
 author = {Schwede, Karl},
 author = {Tucker, Kevin},
 author = {Waldron, Joe},
 author = {Witaszek, Jakub},
 title={Globally $+$-regular varieties and the minimal model program for threefolds in mixed characteristic},
 journal={Publ. Math. Inst. Hautes \'{E}tudes Sci.},
 date={2023},
 issn={0073-8301},
 doi={10.1007/s10240-023-00140-8},
}

\bib{birkar_p}{article}{
   author={Birkar, Caucher},
   title={Existence of flips and minimal models for 3-folds in char $p$},
   language={English, with English and French summaries},
   journal={Ann. Sci. \'{E}c. Norm. Sup\'{e}r. (4)},
   volume={49},
   date={2016},
   number={1},
   pages={169--212},
   issn={0012-9593},
   review={\MR{3465979}},
   doi={10.24033/asens.2279},
}

\bib{birkar_complements}{article}{
   author={Birkar, Caucher},
   title={Anti-pluricanonical systems on Fano varieties},
   journal={Ann. of Math. (2)},
   volume={190},
   date={2019},
   number={2},
   pages={345--463},
   issn={0003-486X},
   review={\MR{3997127}},
   doi={10.4007/annals.2019.190.2.1},
}

\bib{B20}{misc}{
    author={Birkar, Caucher},
    title={On connectedness of non-klt loci of singularities of pairs},
    year={2020},
    note={https://arxiv.org/abs/2010.08226v1, to appear in J. Differential Geom.},
}
\bib{BZ16}{article}{
   author={Birkar, Caucher},
   author={Zhang, De-Qi},
   title={Effectivity of Iitaka fibrations and pluricanonical systems of
   polarized pairs},
   journal={Publ. Math. Inst. Hautes \'{E}tudes Sci.},
   volume={123},
   date={2016},
   pages={283--331},
   issn={0073-8301},
   review={\MR{3502099}},
   doi={10.1007/s10240-016-0080-x},
}



\bib{cascini_tanaka_xu}{article}{
   author={Cascini, Paolo},
   author={Tanaka, Hiromu},
   author={Xu, Chenyang},
   title={On base point freeness in positive characteristic},
   language={English, with English and French summaries},
   journal={Ann. Sci. \'{E}c. Norm. Sup\'{e}r. (4)},
   volume={48},
   date={2015},
   number={5},
   pages={1239--1272},
   issn={0012-9593},
   review={\MR{3429479}},
   doi={10.24033/asens.2269},
}

\bib{filipazzi_svaldi}{article}{
   author={Filipazzi, Stefano},
   author={Svaldi, Roberto},
   title={On the connectedness principle and dual complexes for generalized
   pairs},
   journal={Forum Math. Sigma},
   volume={11},
   date={2023},
   pages={Paper No. e33, 39},
   review={\MR{4580302}},
   doi={10.1017/fms.2023.25},
}


\bib{GNT}{article}{
   author={Gongyo, Yoshinori},
   author={Nakamura, Yusuke},
   author={Tanaka, Hiromu},
   title={Rational points on log Fano threefolds over a finite field},
   journal={J. Eur. Math. Soc. (JEMS)},
   volume={21},
   date={2019},
   number={12},
   pages={3759--3795},
   issn={1435-9855},
   review={\MR{4022715}},
   doi={10.4171/JEMS/913},
}

\bib{hacon_han}{article}{
   author={Hacon, Christopher D.},
   author={Han, Jingjun},
   title={On a connectedness principle of Shokurov--Koll\'{a}r type},
   journal={Sci. China Math.},
   volume={62},
   date={2019},
   number={3},
   pages={411--416},
   issn={1674-7283},
   review={\MR{3905556}},
   doi={10.1007/s11425-018-9360-5},
}

\bib{hacon_witaszek}{article}{
   author={Hacon, Christopher D.},
   author={Witaszek, Jakub},
   title={On the relative minimal model program for fourfolds in positive
   and mixed characteristic},
   journal={Forum Math. Pi},
   volume={11},
   date={2023},
   pages={Paper No. e10, 35},
   review={\MR{4565409}},
   doi={10.1017/fmp.2023.6},
}

\bib{Har77}{book}{
   author={Hartshorne, Robin},
   title={Algebraic geometry},
   note={Graduate Texts in Mathematics, No. 52},
   publisher={Springer-Verlag, New York-Heidelberg},
   date={1977},
   pages={xvi+496},
   isbn={0-387-90244-9},
   review={\MR{0463157}},
}

\bib{Kol92}{book}{
   Editor = {Koll\'ar, J\'anos},
   title={Flips and abundance for algebraic threefolds},
   note={Papers from the Second Summer Seminar on Algebraic Geometry held at
   the University of Utah, Salt Lake City, Utah, August 1991;
   Ast\'{e}risque No. 211 (1992)},
   publisher={Soci\'{e}t\'{e} Math\'{e}matique de France, Paris},
   date={1992},
   pages={1--258},
   issn={0303-1179},
   review={\MR{1225842}},
}

\bib{kollar_singularities}{book}{
   author={Koll\'{a}r, J\'{a}nos},
   title={Singularities of the minimal model program},
   series={Cambridge Tracts in Mathematics},
   volume={200},
   note={With a collaboration of S\'{a}ndor Kov\'{a}cs},
   publisher={Cambridge University Press, Cambridge},
   date={2013},
   pages={x+370},
   isbn={978-1-107-03534-8},
   review={\MR{3057950}},
   doi={10.1017/CBO9781139547895},
}

\bib{KM98}{book}{
   author={Koll\'{a}r, J\'{a}nos},
   author={Mori, Shigefumi},
   title={Birational geometry of algebraic varieties},
   series={Cambridge Tracts in Mathematics},
   volume={134},
   note={With the collaboration of C. H. Clemens and A. Corti;
   Translated from the 1998 Japanese original},
   publisher={Cambridge University Press, Cambridge},
   date={1998},
   pages={viii+254},
   isbn={0-521-63277-3},
   review={\MR{1658959}},
   doi={10.1017/CBO9780511662560},
}

\bib{kollar_witaszek}{misc}{
 author = {Koll\'{a}r, J\'{a}nos},
 author = {Witaszek, Jakub},
 title={Resolution and alteration with ample exceptional divisor},
 year = {2021},
 note = {https://arxiv.org/abs/2102.03162v1},
}

\bib{NT20}{article}{
   author={Nakamura, Yusuke},
   author={Tanaka, Hiromu},
   title={A Witt Nadel vanishing theorem for threefolds},
   journal={Compos. Math.},
   volume={156},
   date={2020},
   number={3},
   pages={435--475},
   issn={0010-437X},
   review={\MR{4053458}},
   doi={10.1112/s0010437x1900770x},
}

\bib{raynaud_gruson}{article}{
   author={Raynaud, Michel},
   author={Gruson, Laurent},
   title={Crit\`eres de platitude et de projectivit\'{e}. Techniques de
   ``platification'' d'un module},
   language={French},
   journal={Invent. Math.},
   volume={13},
   date={1971},
   pages={1--89},
   issn={0020-9910},
   review={\MR{308104}},
   doi={10.1007/BF01390094},
}

\bib{Sho}{article}{
   author={Shokurov, Vyacheslav V.},
   title={Three-dimensional log perestroikas},
   language={Russian},
   journal={Izv. Ross. Akad. Nauk Ser. Mat.},
   volume={56},
   date={1992},
   number={1},
   pages={105--203},
   issn={1607-0046},
   translation={
      journal={Russian Acad. Sci. Izv. Math.},
      volume={40},
      date={1993},
      number={1},
      pages={95--202},
      issn={1064-5632},
   },
   review={\MR{1162635}},
   doi={10.1070/IM1993v040n01ABEH001862},
}


\bib{Wit17}{article}{
   author={Witaszek, Jakub},
   title={On the canonical bundle formula and log abundance in positive
   characteristic},
   journal={Math. Ann.},
   volume={381},
   date={2021},
   number={3-4},
   pages={1309--1344},
   issn={0025-5831},
   review={\MR{4333416}},
   doi={10.1007/s00208-021-02231-5},
}


\end{biblist}
\end{bibdiv}
\end{document}